\documentclass[11pt,leqno]{article}
\usepackage[english]{babel}
\usepackage{amsmath,amssymb,amscd,enumerate,rotate,multirow}
\usepackage{bm,latexsym,mathrsfs}
\usepackage{graphicx}
\usepackage[usenames]{color}
\usepackage{subfigure}
\usepackage{subfig}
\usepackage{pifont}
\usepackage{graphicx,color,psfrag,pgfplots,tikz,xcolor}
\usepackage{amsthm}
\usepackage{bbm}
\usepackage{verbatim}
\usepackage{xspace}
\usepackage{marvosym}

\usepackage{mathtools}
 \mathtoolsset{showonlyrefs = true}

\newcommand{\correz}[1]{{\color{black} #1}}

\usepackage[authoryear]{natbib}

\newtheorem{proposition}{Proposition}

\newcommand{\be}{\begin{equation}}
\newcommand{\ee}{\end{equation}}

\newcommand{\iid}{\stackrel{\mathrm{iid}}{\sim}}
\newcommand{\ind}{\stackrel{\mathrm{ind}}{\sim}}
\newcommand{\e}{\mathrm{e}}
\newcommand{\calB}{\mathcal{B}}

\newcommand{\calL}{\mathcal{L}}
\newcommand{\calN}{\mathcal{N}}

\newcommand{\calX}{\mathcal{X}}

\newcommand{\Prob}{\mathbb{P}}
\newcommand{\E}{\mathbb{E}}
\newcommand{\ff}{\ _2F_1 }
\newcommand{\LL}{\mathcal L}
\newcommand{\rt}{\rightarrow}

\newcommand{\Rea}{{\mathbb{R}}}

\newcommand{\uno}{\mathbb{I}}

\newcommand{\Neps}{N_\varepsilon}
\newcommand{\Peps}{P_\varepsilon}
\newcommand{\Teps}{T_\varepsilon}

\DeclareMathOperator{\Var}{Var}
\DeclareMathOperator{\Cov}{Cov}

\begin{document}

\title{ A priori truncation method for posterior sampling \\ from homogeneous 
normalized completely random measure mixture models}
\author{Raffaele Argiento\thanks{CNR-IMATI \qquad \qquad \ \ \ \Letter \ raffaele@mi.imati.cnr.it} 
\and Ilaria Bianchini\thanks{Politecnico di Milano \ \ \ \Letter \   
\{ilaria.bianchini\}\{alessandra.guglielmi\}@polimi.it} 
\and Alessandra Guglielmi\footnotemark[2]}

\date{\vspace{-5ex}}

\maketitle
\begin{abstract}
\noindent 
This paper adopts a Bayesian nonparametric mixture model where the mixing distribution 
belongs to the wide class of normalized homogeneous completely random measures.
We propose a truncation method for the mixing distribution by discarding the weights of the unnormalized
measure smaller than a threshold. We prove convergence in law 
of  our approximation, provide some theoretical properties and characterize
its  posterior distribution so that a blocked Gibbs sampler is devised.

The versatility of the approximation is illustrated  by two different applications. In the first  
the normalized Bessel random measure, encompassing the Dirichlet process, is introduced;
goodness of fit indexes show its  good performances as mixing measure for
density estimation.
The second describes how to incorporate covariates in the support of the
normalized measure, leading to a linear dependent model for regression and 
clustering.

\vskip6pt
{{\bf \noindent Keywords}: Bayesian nonparametric mixture models
\textbullet~ normalized  completely random measure
\textbullet~ blocked Gibbs sampler
\textbullet~ finite dimensional approximation
\textbullet~ a priori truncation method}
\end{abstract}

\section{Introduction}
\label{sec:1}
One of the livelier topic in Bayesian Nonparametrics concerns mixtures of parametric densities where the mixing measure is an almost surely discrete random probability measure. 
The basic model is what is known now as Dirichlet process mixture model, appeared first in \cite{Lo84}, 
where the mixing measure is indeed the Dirichlet process. Dating back to \cite{IshJam01} and \cite{lij_etal05NIG}, many alternative mixing measures have been proposed; the former paper replaced the Dirichlet process with  stick-breaking random probability measures, while the latter focused on normalized completely random measures.  

These hierarchical mixtures
play a pivotal role in modern Bayesian Nonparametrics, since their potentialities range within many applications. 
Indeed, they can easily be exploited in very different contexts: for instance, graphical models, topic modeling or biological applications.
Their popularity is mainly due to the high flexibility in density estimation problems as well as in clustering, which is naturally embedded in the model. 

Often the Dirichlet Process prior is employed as mixing measure because of its mathematical 
and computational tractability: however, in 
some statistical applications, clustering induced by the Dirichlet process may be restrictive. 
In fact, it is well-know that the latter allocates observations to clusters with probabilities depending only on the 
cluster sizes, leading to the 
"the rich gets richer" behavior. 
Within some classes of more general processes, as, for instance,  stick-breaking and normalized processes, 
the probability of allocating an observation to a specific cluster depends also on extra parameters, 
as well as on the  number of groups and on the cluster's size. 
We refer to \cite{ArgGug15} for a recent review of 
state of art on Bayesian nonparametric mixture models and clustering.

Since, when dealing with nonparametric mixtures, the posterior inference involves 
an infinite-dimensional parameter, this may lead to computational issues; this limit prevents applied statisticians 
from exploiting models beyond Dirichlet process mixtures when dealing with modern real-life applications.
However, there is a recent and lively literature  focusing mainly on two different classes of MCMC algorithms, 
namely \textit{marginal} and \textit{conditional} Gibbs samplers.
The former integrate out the
infinite dimensional parameter (i.e. the random probability),
resorting to generalized Polya urn schemes;
see \cite{FavTeh13} or \cite{LomFav14}. 
The latter include the nonparametric mixing measure in the state space of the Gibbs sampler, 
updating it as a component of the algorithm;
this group includes the slice sampler \cite[see][]{GriWal11}.
Among conditional algorithms there are truncation methods, where the infinite parameter (i.e. the mixing measure) 
is approximated by truncation of the infinite sums defining the process, either a posteriori 
\citep{Arg_etal10,bar_etal13} or a priori \citep{peps,griffin2013adaptive}.

In this work we introduce an almost surely finite dimensional class of random probability measures that approximates the 
wide family of homogeneous normalized completely random measures \citep{Regazzini_etal2003,Kingman75}; we use this class as the building block in mixture models 
and provide a simple but general algorithm to perform posterior inference.
Our approximation is based on the constructive definition of the weights of the completely random measure 
as the points of a Poisson process on $\mathbb{R}^+$. In particular, we consider only points larger than a threshold $\varepsilon$, controlling the 
degree of approximation. 
The construction given here generalizes \cite{peps} where the particular class of normalized generalized gamma processes
was considered.  
Conditionally on $\varepsilon$,
our process is finite dimensional either  a priori and a posteriori.

As detailed later, the two main ingredients to build a normalized completely random measure are  $\rho(s)$, $s>0$, 
the intensity of the Poisson process determining the weights of the measure  on the one hand, and 
the so-called centering measure $P_0(\cdot)$, characterizing the locations of the measure, on the other.
Here we illustrate two applications. In the first, a new choice for $\rho$ is proposed: the Bessel 
intensity function, that, up to our knowledge, has never been applied in a statistical framework, but 
in finance \citep[see][for instance]{barndorff00}. On the other hand, we fix the centering measure
$P_0$ to be the normal inverse-gamma, a conjugate choice when the kernel is  Gaussian. 
We call this new process normalized Bessel random measure.
In the second application, we set $\rho$ to be the well-known generalized gamma intensity and consider 
a centering measure $P_{0\textbf{x}}$ depending on 
on a set of covariates $\textbf{x}$, 
yielding a linear dependent normalized completely random measure.
For a recent survey on dependent nonparametric processes 
in the Statistics and Machine Learning literature see \cite{FotiWill2015}.

In this paper, since 
the main objective is the approximation of the nonparametric process arisen from the normalization of  
completely random measures, we fix $\varepsilon$ to a small value. 
However, it is worth mentioning that it is possible to elicit a prior for $\varepsilon$,
but the computational cost might greatly increase for some $\rho$.

The main achievements of this works  can be summarized as follows: first we show that, for 
$\varepsilon$ going to zero,
the  finite dimensional $\varepsilon$-approximation of homogeneous normalized completely random measures  
converges to its 
infinite dimensional counterpart, and compute its prior moments (Sections 3 and 4). 
Then we  provide a Gibbs sampler for the $\varepsilon$ -approximation hierarchical mixture model (Section 5). Section 6.1 is devoted to the introduction of the normalized 
normalized Bessel random measure, 
and some of its properties;
on the other hand, Section 6.2 discusses an application of the $\varepsilon$-Bessel mixture models to both simulated and real data. 
Section 7 defines the linear dependent $\varepsilon$-NGGs, and consider
 linear dependent $\varepsilon$-NGG mixtures to fit the AIS data set. 
To complete the set-up of the paper, Section \ref{sec:homogen} is devoted to a summary of basic notions about homogeneous NRMIs, 
and Section \ref{sec:last} contains a conclusive discussion.

\section{Preliminaries on homogeneous normalized completely random measures}
\label{sec:homogen}
Let us briefly recall the definition of a homogeneous normalized completely random measure. 
Let $\Theta\subset\Rea^m$ for some positive integer $m$. 
A random measure $\mu$ on  $\Theta$ is  completely random  if for any finite sequence $B_1,B_2,\dots,B_k$ of disjoint sets in $\mathcal{B}(\Theta)$,
$\mu(B_1),  \mu(B_2),\dots,\mu(B_k)$ are independent. A  purely atomic completely random measure is defined \citep[see][Section 8.2]{Kingman93} by 
 $ \mu(\cdot)=\sum_{j\ge1}J_j \delta_{\tau_j}(\cdot)$, where the  $\{(J_j,\tau_j)\}_{j\ge1}$ are the points of a Poisson process on $\Rea^{+}\times \Theta$.
We denote by $\nu(ds,d\tau)$ the intensity of the mean measure of such a Poisson process. A completely random measure is homogeneous if $\nu(ds,d\tau)=\rho(s)ds \kappa P_0(d\tau)$, where $\rho(s)$ is the density of a non-negative measure on $\Rea^{+}$, while $\kappa P_0$ is a
finite measure on $\Theta$ with total mass $\kappa>0$. If  $\mu $ is homogeneous,  the support points,  that is $\{\tau_j\}$,  and the jumps of $\mu$,
 $\{J_j\}$, are independent, and the $\tau_j$'s are independent identically distributed (iid) random variables from $P_0$, while $\{J_j\}$ are the  points of a Poisson
process on $\Rea^+$ with mean intensity $\rho$. Furthermore, we assume that $\rho$ satisfies the following regularity conditions:
\begin{equation}
  \begin{array}[h!]{ccc}
  \displaystyle \int_{0}^{+\infty} \min\{1,s\} \rho(s)ds<\infty &
   \text{and} &
   \displaystyle  \int_{0}^{+\infty}  \rho(s)ds=+\infty,
  \end{array}
  \label{eq:regolarita}
\end{equation}
so that, if $T:=\mu(\Theta)=\sum_{j\ge 1}^{}J_j$, $\Prob(0<T<+\infty)=1$. 
Recall that the distribution of $T$ is uniquely determined by its Laplace transform, given by:
\begin{equation}
\label{eq:LaplaceTransf}
\E(\e^{-\lambda T})= \exp{\{ -\kappa \int_0^{+\infty}(1-\e^{-\lambda s}) \rho(s)ds \}},\quad \lambda\ge 0.
\end{equation}
Therefore, a random probability measure (r.p.m.) $P$ can be defined through normalization of $\mu$:
\begin{equation}
  P:=\frac{\mu}{\mu (\Theta)}=\sum_{j=1}^{+\infty}
   \frac{J_j}{T}\  \delta_{\tau_j}=\sum_{j=1}^{+\infty}P_j\delta_{\tau_j}\ .
  \label{eq:NRMI}
\end{equation}
We refer to $P$ in \eqref{eq:NRMI} as a (homogeneous) normalized completely random measure with parameter $(\rho, \kappa P_0)$. As an alternative notation, following 
\cite{jamesetal09}, $P$ is referred as a homogeneous normalized measure with independent increments. 
 The definition of normalized completely random measures appeared in 
\cite{Regazzini_etal2003} first. An alternative construction of normalized
completely random measures can be given in terms of 
Poisson-Kingman models as in  \cite{Pitman03}.

\section{$\varepsilon$-approximation of normalized completely random measures}
\label{sec:epsHNRMI}
The goal of this section is the definition of a finite dimensional random
probability measure that is an approximation of a general normalized
completely random measure 
with Levy's intensity given by $\nu(ds, d\tau)=\rho(ds)\kappa P_0(d\tau)$, introduced above. 

First of all, by the Restriction Theorem for Poisson processes,  
for any  $\varepsilon >0$, all the jumps $\{J_j\}$ of $\mu$ larger than a threshold 
$\varepsilon$ are still a Poisson process, with mean intensity $\gamma_\varepsilon(s):=\kappa \rho(s)\uno_{(\varepsilon,+\infty)}(s)$. 
Moreover, the total number of these points is Poisson distributed, i.e. $N_{\varepsilon}\sim \mathcal{P}_0(\Lambda_\varepsilon)$ where
\begin{equation*}
  \Lambda_\varepsilon:=\gamma_{\varepsilon}(\mathbb{R}^+)=\kappa\int_{\varepsilon}^{+\infty} \rho(s)ds. 
 \end{equation*}
Since $ \Lambda_\varepsilon<+\infty$ for any $\varepsilon>0$ thanks to the regularity conditions \eqref{eq:regolarita}, $N_{\varepsilon}$ is almost surely finite. 
In addition, conditionally to $N_\varepsilon$, the points
$\{J_1,\dots,J_{N_\varepsilon}\}$ are iid from the density 
\begin{equation}
\label{eq:rhoeps}
  \rho_\varepsilon(s)=\dfrac{\gamma_\varepsilon(s)}{\Lambda_\varepsilon} = \dfrac{\kappa \rho(s)}{\Lambda_\varepsilon}\uno_{(\varepsilon,+\infty)}(s),
\end{equation}
thanks to the relationship between Poisson and Bernoulli processes; see, for instance, \cite{Kingman93}, Section~2.4.
However, in this case, while $\Prob(\sum_{j=1}^{N_\varepsilon} J_j <\infty)=1$, the condition on the right hand side of 
\eqref{eq:regolarita} is not satisfied, so that 
$\Prob(\sum_{j=1}^{N_\varepsilon} J_j=0)>0$, or, in other terms, $\Prob(N_\varepsilon=0)>0$ for any $\varepsilon>0$. 
For this reason we consider $\Neps+1$ iid points  $\{J_0,J_1,\dots,J_{N_\varepsilon}\}$ from  $\rho_\varepsilon$ and define
the completely random measure  $\mu_\varepsilon(\cdot)=\sum_{j=0}^{N_\varepsilon} J_j\delta_{\tau_j}(\cdot)$, as well as  its normalized counterpart:
\begin{equation}
P_\varepsilon(\cdot)=\sum_{j=0}^{\Neps}P_{j}\delta_{\tau_j}(\cdot) = \sum_{j=0}^{\Neps}\frac{J_j}{T_\varepsilon}\delta_{\tau_j}(\cdot),
\label{eq:Pvarepsilon}
\end{equation}
where $T_\varepsilon=\sum_{j=0}^{\Neps} J_j$,  $\tau_j\iid P_0$, $\{ \tau_j\}$ and  $\{ J_j\}$ independent.
We denote $\Peps$ in \eqref{eq:Pvarepsilon} by $\varepsilon$-NormCRM and write 
$P_\varepsilon \sim \varepsilon-NormCRM(\rho, \kappa P_0)$. When $\rho_\varepsilon(s)=1/(\omega^\sigma\Gamma(-\sigma,\omega\varepsilon) )
  s^{-\sigma-1}\e^{-\omega s}$, $s>\varepsilon$, $\Peps$ is the $\varepsilon$-NGG process introduced in \cite{peps}, with parameter $(\sigma,\kappa,P_0)$,  $0\leq \sigma\leq 1$, $\kappa\geq 0$.

Both the infinite and finite dimensional processes defined in \eqref{eq:NRMI} and  \eqref{eq:Pvarepsilon}, respectively, 
belong to the wide class of species sampling models, deeply investigated 
in \cite{Pitman96}, and we use some of the results there to derive ours. 
Let $\left(\theta_1,\dots,\theta_n\right)$ be a sample from  
\eqref{eq:NRMI} or \eqref{eq:Pvarepsilon} (or more generally, from a species
sampling model); 
since it is a sample from a discrete  probability,
it induces a random partition ${\bm p}_n:=\{C_1,\dots,C_k\}$ on the set  $\mathbb{N}_n:=\{1,\dots,n\}$ where  $C_j = \{i : \theta_i = \theta^*_j\} $ for $j=1,\dots,k$.  
If $\#C_i=n_i$ for $1\le i \le k$, the marginal law of $\left(\theta_1,\dots,\theta_n\right)$ has unique characterization: 
$$\LL({\bm p}_n,\theta^*_1,\dots,\theta^*_k)=p(n_1,\dots,n_k) \prod_{j=1}^{k}\LL(\theta^*_j),$$ 
where $p$ is the exchangeable partition probability function (eppf) associated to the random probability. 
The eppf $p$ is a probability law on the set of the partitions of $\mathbb{N}_n$.
The following proposition provides an expression for the eppf of a general  $\varepsilon$-NormCRM.
\begin{proposition}
Let  $(n_1,\dots,n_k)$ be a vector of positive integers such that $\sum_{i=1}^kn_i=n$. Then, the eppf associated with a $\Peps\sim\varepsilon$-NormCRM$(\rho,\kappa P_0)$ is 
\begin{equation}
\begin{split}
p_{\varepsilon}(n_1,\dots,n_k) & =  
\int_{0}^{+\infty}\left[\dfrac{u^{n-1}}{\Gamma(n)}\dfrac{(k+\Lambda_{\varepsilon,u})}{\Lambda_{\varepsilon}}e^{\left(\Lambda_{\varepsilon,u}-\Lambda_\varepsilon\right)}
\prod_{i=1}^k \int_{\varepsilon}^{+\infty} \kappa s^{n_i}\e^{-us}\rho(s)ds\right]du  \end{split}
    \label{eq:eppf_expr}
\end{equation}
where
\begin{equation}
\label{eq:lambda_epsu}
  \Lambda_{\varepsilon,u}:=\kappa \int_{\varepsilon}^{+\infty}\e^{-us}\rho(s)ds, \quad u\ge 0.
  \end{equation}
\label{prop:eppfPeps}
\end{proposition}

\begin{proof}
We have
\begin{eqnarray}
p_{\varepsilon}(n_1,\dots,n_k)=\sum_{\Neps=0}^{+\infty} p_{\varepsilon}(n_1,\dots,n_k|\Neps)
  \frac{\Lambda_{\varepsilon}^{\Neps}}{\Neps !}\e^{-\Lambda_{\varepsilon}},
\label{eq:eppf_sum}
\end{eqnarray}
since $N_\varepsilon\sim Poi(\Lambda_\varepsilon)$. Then, equation (30) in \cite{Pitman96}
yields
\begin{equation*}
p_{\varepsilon}(n_1,\dots,n_k|\Neps)=  \uno_{ \{1,\dots,\Neps+1\} }(k)
\sum_{j_1,\dots,j_k}^{}\E\left(
  \prod_{i=1}^{k}P_{j_i}^{n_i} \right),
\end{equation*} 
where the vector $(j_1,\dots,j_k)$ ranges over all permutations of $k$ elements in $\{0,\dots,N_\varepsilon\}$. Then, using the gamma function identity,
 $1/T_\varepsilon^n=\int_{0}^{+\infty}1/\Gamma(n)u^{n-1} e^{-u T_\varepsilon} du$,
we have: 
 \begin{equation*}
  \begin{split}
  & p_{\varepsilon}(n_1,..,n_k|\Neps) =\uno_{ \{1,\dots,\Neps+1\} }(k)\sum_{j_1,\dots,j_k}    \int_{}^{}
   \prod_{i=1}^{k}\frac{J_{j_i}^{n_i}}{T_{\varepsilon}^{n_i}} \mathcal{L}(dJ_0,\dots,dJ_{\Neps})\\
 \quad &= \uno_{ \{1,\dots,\Neps+1\} }(k) \sum_{j_1,\dots,j_k} \int_{0}^{+\infty} du \left( \frac{1}{\Gamma(n)}u^{n-1}  
 \prod_{i=1}^{k}\int_{0}^{+\infty}J_{j_i}^{n_i}\e^{-J_{j_i}u}\rho_{\varepsilon}(J_{\correz{j_i}})dJ_{j_i} \right.  \\ 
  &  \qquad \qquad \qquad\qquad \qquad \qquad \qquad  \qquad \qquad \qquad \qquad\times \left.
  \prod_{j\notin\{j_1,\dots,j_k\}}^{}\int_{0}^{+\infty}\e^{-J_{j}u}\rho_{\varepsilon}(\correz{J_j})dJ_j \right).
 \end{split}
 \end{equation*}
Now, by the definition of $\rho_\varepsilon$ in \eqref{eq:rhoeps} and adopting the notation $\tilde \rho(s):= \kappa \rho(s)$, it is straightforward to see that
 \begin{equation*}
  \begin{split}
   & p_{\varepsilon}(n_1,..,n_k|\Neps) =  \uno_{ \{1,\dots,\Neps+1\} }(k) \sum_{j_1,\dots,j_k} \int_{0}^{+\infty} du \left( \frac{1}{\Gamma(n)}u^{n-1} 
   \prod_{i=1}^{k}\int_{\varepsilon}^{+\infty}J_{j_i}^{n_i}\e^{-J_{j_i}u} \ \dfrac{\tilde\rho(J_{j_i})}{\Lambda_\varepsilon}dJ_{j_i}  \right.\\ 
  &  \qquad \qquad  \qquad \qquad\qquad \qquad \qquad \qquad  \qquad \qquad \qquad \qquad \times \left.
  \prod_{j\notin\{j_1,\dots,j_k\}}^{}\int_{\varepsilon}^{+\infty}\e^{-J_{j}u} \ \dfrac{\tilde\rho(J_{j})}{\Lambda_\varepsilon}dJ_{j} \right) .
  \end{split}
 \end{equation*}
By $\eqref{eq:eppf_sum}$, we have
  \begin{equation*}
  \begin{split}
  & p_\varepsilon(n_1, \dots, n_k)=\sum_{N_\varepsilon=0}^{+\infty}\uno_{ \{1,\dots,\Neps+1\} }(k) \sum_{j_1,\dots,j_k} \int_{0}^{+\infty} du \left( \frac{u^{n-1}}{\Gamma(n)} 
   \dfrac{1}{\Lambda_\varepsilon^k} \prod_{i=1}^{k}\int_{\varepsilon}^{+\infty}J_{j_i}^{n_i}\e^{-J_{j_i}u} \tilde\rho(J_{j_i})dJ_{j_i}  \right.\\ 
   &     \qquad\qquad \qquad \qquad \qquad  \qquad \qquad \qquad \qquad \times \left.
  \dfrac{1}{\Lambda_\varepsilon^{N_\varepsilon+1-k}} \prod_{j\notin\{j_1,\dots,j_k\}}^{}\int_{\varepsilon}^{+\infty}\e^{-J_{j}u} \ 
  \tilde\rho(J_{j})dJ_{j} \right) \dfrac{\Lambda_\varepsilon^{N_\varepsilon}}{N_\varepsilon!}e^{-\Lambda_\varepsilon}\\
 &\qquad  \qquad \qquad = \sum_{N_\varepsilon=0}^{+\infty}\uno_{ \{1,\dots,\Neps+1\} }(k) \dfrac{e^{-\Lambda_\varepsilon}}{\Lambda_\varepsilon N_\varepsilon!}  
 \int_{0}^{+\infty} \Biggl\{\frac{u^{n-1}}{\Gamma(n)}\left(\int_{\varepsilon}^{+\infty}\e^{-J_{j}u}\tilde\rho(J_{j})dJ_{j} \right)^{N_\varepsilon+1-k}
\\
&  \qquad \qquad \qquad   \qquad\qquad \qquad \qquad \qquad  \qquad \qquad \qquad  \times 
\sum_{j_1,\dots,j_k} \left(  \prod_{i=1}^{k}\int_{\varepsilon}^{+\infty}J_{j_i}^{n_i}\e^{-J_{j_i}u}\tilde\rho(J_{j_i})dJ_{j_i} \right)\Biggr\} du. 
  \end{split}
\end{equation*}
Denoting by  $N_{na}:=N_\varepsilon+1-k$ the number of
non-allocated jumps, we get
  \begin{equation*}
  \begin{split}
  & p_\varepsilon(n_1, \dots, n_k)
= \int_{0}^{+\infty} \Biggl\{\frac{u^{n-1}}{\Gamma(n)} \dfrac{e^{-\Lambda_\varepsilon}}{\Lambda_\varepsilon} 
 \sum_{N_\varepsilon=0}^{+\infty} 
 \dfrac{N_\varepsilon+1}{(N_\varepsilon+1-k)!}\uno_{ \{1,\dots,\Neps+1\} }(k)
 \\
 & \qquad  \qquad \qquad\qquad  \qquad  \times  \left(\int_{\varepsilon}^{+\infty}\e^{-J_{j}u}\tilde\rho(J_{j})dJ_{j} \right)^{N_\varepsilon+1-k}
 \left(  \prod_{i=1}^{k}\int_{\varepsilon}^{+\infty}J_{j_i}^{n_i}\e^{-J_{j_i}u}\tilde\rho(J_{j_i})dJ_{j_i} \right)  \Biggr\} du \\
&= \int_{0}^{+\infty}  \frac{u^{n-1}}{\Gamma(n)} \dfrac{e^{-\Lambda_\varepsilon}}{\Lambda_\varepsilon} 
 \prod_{i=1}^{k}\int_{\varepsilon}^{+\infty}J_{j_i}^{n_i}\e^{-J_{j_i} u}\tilde\rho(J_{j_i})dJ_{j_i}
 \sum_{N_{na}=0}^{+\infty} \Lambda_{\varepsilon,u}^{N_{na}}
 \dfrac{N_{na}+k}{N_{na}!}\uno_{ \{1,\dots,\Neps+1\} }(k)\ du.
  \end{split}
\end{equation*}
Since the last  summation adds up to $e^{\Lambda_{\varepsilon,u}}\left(\Lambda_{\varepsilon,u}+k\right)$, the 
$p_\varepsilon(n_1, \dots, n_k)$ and the right hand-side of \eqref{eq:eppf_expr} coincide.
\end{proof}

A result concerning the eppf of a generic normalized (homogeneous) completely random measure can be readily obtained from \cite{Pitman03}, formulas (36)-(37): 
\begin{equation}
\label{eq:eppf_hnrmi}
 p(n_1, \dots, n_k) = \int_0^{+\infty}\dfrac{u^{n-1}}{\Gamma(n)}\e^{\kappa\int_0^{+\infty}(\e^{-us}-1)\rho(s)ds }
\left( \prod_{i=1}^k\int_0^{+\infty}\kappa s^{n_i}\e^{-us}\rho(s)ds\right) du.
\end{equation}
Now we are ready to show that the eppf of  \eqref{eq:Pvarepsilon}
converges pointwise to that of the corresponding (homogeneous) normalized
completely random measure  \eqref{eq:NRMI} 
when $\varepsilon \to 0$.
\begin{proposition} 
\label{prop:conv_eppf}
Let $p_\varepsilon(\cdot)$ be the eppf of a $\varepsilon-$NormCRM$(\rho, \kappa P_0)$. 
Then for any sequence $n_1,\dots,n_k$ of positive integers with $k>0$ and
$\sum_{i=1}^{k}n_i=n$,
\begin{equation}
  \lim_{\varepsilon\rightarrow0}p_{\varepsilon}(n_1,\dots,n_k)=p_0(n_1,\dots,n_k),  
\label{eq:conv_eppf}
\end{equation}
where $p_0(\cdot)$ is the eppf of the NormCRM$(\rho,\kappa P_0)$ as in \eqref{eq:eppf_hnrmi}. 
\end{proposition}

\begin{proof}
By Proposition~\ref{prop:eppfPeps} 
\begin{equation*}
  p_{\varepsilon}(n_1,\dots,n_k)=\int_{0}^{+\infty}f_{\varepsilon}(u;n_1,\dots,n_k)du 
\end{equation*}
where 
\begin{equation}\label{eq:integr_eppf}
f_{\varepsilon}(u; n_1,\ldots,n_k)= \dfrac{u^{n-1}}{\Gamma(n)}\dfrac{(k+\Lambda_{\varepsilon,u})}{\Lambda_{\varepsilon}}e^{\left(\Lambda_{\varepsilon,u}-\Lambda_\varepsilon\right)}
\prod_{i=1}^k \int_{\varepsilon}^{+\infty} \kappa s^{n_i}\e^{-us}\rho(s)ds, \quad u>0.
\end{equation}
On the other hand, the eppf of a NormCRM$(\rho, \kappa P_0)$ can be written as 
\begin{equation*}
  p_0(n_1,\dots,n_k)=\int_{0}^{+\infty} f_0(u;n_1,\dots,n_k)du, 
\end{equation*}
where 
\begin{equation*}
f_0(u;n_1,\dots,n_k)=\dfrac{u^{n-1}}{\Gamma(n)}\exp\left\{\kappa\int_0^{+\infty}(\e^{-us}-1)\rho(s)ds \right\}
 \prod_{i=1}^k\int_0^{+\infty}\kappa s^{n_i}\e^{-u s}\rho(s), \quad u>0.
\end{equation*}
 We first show that  
\begin{equation}\label{eq:limite}
  \lim_{\varepsilon\rightarrow0}f_{\varepsilon}(u;n_1,\dots,n_k)=
f(u;n_1,\dots,n_k)\quad \textrm{ for any } u>0.
\end{equation}
In particular, we have that 
$$ \lim_{\varepsilon \to 0}\int_{\varepsilon}^{+\infty}s^{n_i}\e^{-u s}\rho(s)ds =  \int_{0}^{+\infty}s^{n_i}\e^{-u s}\rho(s)ds  $$
and 
$$  \lim_{\varepsilon \to 0}\e^{\Lambda_{\varepsilon,u}-\Lambda_\varepsilon}= \exp\left\{\kappa \int_0^{+\infty}(\e^{-us}-1)\rho(s)ds \right\},$$
being this limit finite for any $u>0$.
Using standard integrability criteria, it is straightforward to check that, for any $u>0$,
 $\lim_{\varepsilon \to 0}\Lambda_{\varepsilon,u}=\lim_{\varepsilon \to 0}\Lambda_{\varepsilon}=+\infty$ and they are equivalent infinite, i.e. 
$$ \lim_{\varepsilon \to 0} \dfrac{k + \Lambda_{\varepsilon,u}}{\Lambda_\varepsilon}= \lim_{\varepsilon \to 0} \dfrac{ \Lambda_{\varepsilon,u}}{\Lambda_\varepsilon} =1.$$ 
We can therefore conclude that \eqref{eq:limite} holds true.

The rest of the proof follows as in the second part of the proof of Lemma~2  in \cite{peps}, where we prove that 
$(i)$ $\lim_{\varepsilon\rightarrow0}\sum_{\mathcal{C} \in \Pi_n}^{}p_{\varepsilon}(n_1,\dots,n_k) =1$;
$(ii)$ $\liminf_{\varepsilon\rightarrow0} p_\varepsilon(n_1,\dots,n_k)=p_0(n_1,\dots,n_k)$
 for all $\mathcal{C}=(C_1$, $\dots,C_k) \in \Pi_n$, the set of all partitions of $\{1,2,\ldots,n\} $;
$(iii)$ $\sum_{\mathcal{C} \in \Pi_n}^{} p_0(n_1,\dots,n_k)=1$. By Lemma~1  in \cite{peps}, equation $\eqref{eq:conv_eppf}$ follows.

\end{proof}

Convergence of the sequence of eppfs yields convergence of the sequences of $\varepsilon$-NormCRMs, generalizing a 
result obtained for $\varepsilon$-NGG processes. 
\begin{proposition}
Let $P_\varepsilon$ be a $\varepsilon$-NormCRM$(\rho, \kappa P_0)$, for any $\varepsilon >0$. Then 
\begin{equation*}
P_\varepsilon \stackrel{d}\rightarrow P \textrm{ as } \varepsilon \rightarrow 0, 
\end{equation*}
where $P$ is a NormCRM$(\rho, \kappa P_0)$. 
Moreover, as $\varepsilon\rt +\infty$,  $P_\varepsilon \stackrel{d}\rightarrow \delta_{\tau_0}$, where $\tau_0\sim P_0$.
\end{proposition}
\begin{proof}
Since $P_\varepsilon$ is a proper species sampling model,  $p_{\varepsilon}$ defines a probability law on the sets of all partitions of  
$\{1,\dots,n\}$, for any positive integer $n$;  let $(N_1^\varepsilon,\ldots,N_k^\varepsilon)$ denote the sizes  of the blocks (in order of appearance) of the random partition $C_{\varepsilon,n}$ defined by $p_\varepsilon$, for any $\varepsilon \geq 0$. The probability distributions of $\{(N_1^\varepsilon,\ldots,N_k^\varepsilon),\varepsilon\geq 0 \}$ are proportional to the values of $p_\varepsilon$ (for any $\varepsilon\geq 0$) in (2.6) in \cite{Pitman06}.
Hence, by Proposition~\ref{prop:conv_eppf}, for any $k=1,\ldots,n$  and any  $n$, 
\begin{equation*}
(N_1^\varepsilon,\ldots,N_k^\varepsilon) \stackrel{d}\rightarrow(N_1^0,\ldots,N_k^0)  \textrm{ as } \varepsilon \rightarrow 0,
\end{equation*}
where  $(N_1^0,\ldots,N_k^0)$  denote the sizes  of the blocks 
of the random partition $C_{\varepsilon,n}$ defined by $p_0$, the eppf of a NormCRM$(\rho, \kappa P_0)$ process. 
By formula (2.30) in \cite{Pitman06},  we have 
\begin{equation*}
\begin{CD}
 \left(\frac{N_j^\varepsilon}{n}\right) @>d>{n\rightarrow +\infty}> (\tilde P_j^\varepsilon)  \\
@V{\varepsilon \rightarrow0}VdV @.\\ 
\left(\frac{N_j^0}{n}\right) @> d >{n\rightarrow +\infty} > (\tilde P_j)
\end{CD}
\end{equation*}
where $P_j^\varepsilon$ and $\tilde P_j$ are the $j$-th weights of a $\varepsilon$-NormCRM and a NormCRM process (with parameters $(\rho,\kappa P_0)$), respectively. 
We prove that 
\begin{equation*}
 \begin{CD}
\sum_{j\ge 0} \tilde P_j^\varepsilon\delta_{\tau_j} @>d>{n\rightarrow +\infty}>  
\sum_{j\ge 0} \tilde P_j^\varepsilon \delta_{\tau_j},
\end{CD}
\end{equation*}
where  $ \tau_0,\tau_1,\tau_2,\ldots$  are iid from $P_0$
and this ends the first part of the Proposition. 

Convergence as $\varepsilon\rt +\infty$ is straightforward as well. In fact,  when $\varepsilon$ increases to $+\infty$, there are no jumps to consider in \eqref{eq:Pvarepsilon} but the extra $J_0$, so that $P_\varepsilon$ degenerates on $\delta_{\tau_0}$. 
\end{proof}

Let $\bm \theta=(\theta_1,\dots,\theta_n)$ be a sample from $P_\varepsilon$, a $\varepsilon$-NormCRM$(\rho, \kappa P_0)$ as defined in \eqref{eq:Pvarepsilon}, and let ${\bm\theta}^*=(\theta^*_1,\dots,\theta^*_k)$  be the (observed) distinct values in  ${\bm\theta}$. 
We denote by \emph{allocated} jumps of the process the values $P_{l^*_1},P_{l^*_2}, \dots, P_{l^*_k}$ in \eqref{eq:Pvarepsilon} such that there exists a corresponding location for which $\tau_{l^*_i}=\theta^*_i$, $i=1, \dots, k$.
The remaining values are \emph{non-allocated} jumps. We use the superscript $(na)$ for random variables related to \emph{non-allocated} jumps.  We also introduce the random variable $U:=\Gamma_n/\Teps$, where $\Gamma_n\sim gamma(n,1)$, being $\Gamma_n$ and $\Teps$ independent. 

\begin{proposition}
\label{prop:Posterior}
  If $P_\varepsilon$ is an $\varepsilon-$NormCRM$(\rho, \kappa P_0)$, then the conditional distribution of $P_{\varepsilon}$, 
  given ${\bm\theta}^*$ and $U=u$, verifies the distributional equation
  \begin{equation*}
    P_{\varepsilon}^*(\cdot) \stackrel{d}=
w  P_{\varepsilon,u}^{(na)}(\cdot)+(1-w)\sum_{j=1}^{k}P^{(a)}_j\delta_{\theta^*_k}(\cdot)
  \end{equation*}
  where 
  \begin{enumerate}
    \item \label{prop:noalloc}
      $P_{\varepsilon,u}^{(na)}(\cdot)$, the process of non-allocated jumps, is distributed 
      as a $\varepsilon-$NormCRM$(\e^{-u \cdot}\rho(\cdot) , \kappa P_0)$, given that exactly
      $N_{na}$ jumps of the process were obtained, and the posterior law of  $N_{na}$ is 
      $$\frac{\Lambda_{\varepsilon,u}}{k+\Lambda_{\varepsilon,u}} \mathcal{P}_1(\Lambda_{\varepsilon,u}) + 
      \frac{k}{k+\Lambda_{\varepsilon,u}}
      \mathcal{P}_0(\Lambda_{\varepsilon,u}),$$
being $\Lambda_{\varepsilon,u}$ as defined in \eqref{eq:lambda_epsu}, and  denoting $\mathcal{P}_i(\lambda)$ the shifted Poisson distribution on $\{i,i+1,i+2,\ldots \}$ with mean $i+\lambda$, $i=0,1$; 
    
\item  \label{prop:alloc} the allocated jumps $\{P_1^{(a)},\dots,P_k^{(a)}\}$ \emph{associated} to the fixed points of discontinuity ${\bm\theta}^*=(\theta^*_1,\dots,\theta^*_k)$ of $P^*_\varepsilon$
      are obtained by normalization of 
      $J_j^{(a)} \stackrel{ind}{\sim} J_j^{n_i}\e^{-u J_j}e^{-uJ_j}\rho(J_j)\uno_{(\varepsilon, +\infty)}(J_j)$, for $j=1\dots,k$;

\item \label{prop:indep} $P_{\varepsilon,u}^{(na)}(\cdot)$ and $\{J_1^{(a)},\cdots,J_k^{(a)}\}$ are independent, conditionally to 
${\bm l}^*=(l^*_1,\ldots,l^*_k )$, the vector of locations of the \emph{allocated} jumps;  
 
\item  \label{prop:pesomist}  $w$ is defined as 0 when $N_{na}=0$, 
otherwise $w=T_{\varepsilon,u}/ (  T_{\varepsilon,u}+\sum_{j=1}^{k}
J_j^{(a)})$. $T_{\varepsilon,u}$ is the total sum of the jumps in
representation of $P_{\varepsilon,u}^{(na)}(\cdot)$ as in
$\eqref{eq:Pvarepsilon}$;

\item  \label{prop:margU} the posterior law of $U$ given  ${\bm\theta}^*$ has density on the positive real given by
  \begin{equation*}
    f_{U|\bm \theta^{*}}(u|\bm \theta^{*})\propto 
   u^{n-1} \e^{\Lambda_{\varepsilon,u}-\Lambda_\varepsilon}\dfrac{\Lambda_{\varepsilon,u}+k}{\Lambda_\varepsilon}
   \prod_{i=1}^k\int_\varepsilon^{+\infty}\kappa s^{n_i}\e^{-u s}\rho(s) ds, \quad u>0.
  \end{equation*}
\end{enumerate}

\end{proposition} 
This  proposition is the  ``finite dimensional'' counterpart of Theorem~1 in  \cite{jamesetal09}.

\begin{proof}
The first steps of the proof are the same as in the proof of Proposition~2 in \cite{peps}; in particular, the joint law of ${\bm \theta}, u, \Peps$, $\calL({\bm\theta}, u, P_{\varepsilon})$, is as in (16) in \cite{peps}. 
The conditional distribution of $\Peps$, given $U=u$ and $\bm\theta$, is as follows:
\begin{equation}
\mathcal{L}(P_{\varepsilon}|u,\bm\theta)= \mathcal{L}({\bm \tau},{\bm J},N_\varepsilon|u,\bm\theta)
=  \mathcal{L}({\bm \tau},{\bm J}|N_\varepsilon, u,\bm\theta)  \mathcal{L}(N_\varepsilon|u,\bm\theta),
\label{eq:post_peps}
\end{equation}
where the second factor in the right hand side is proportional to 
\begin{align*}
 \mathcal{L}(N_\varepsilon, u,\bm\theta) &= \int d J_0 \ldots d J_{\Neps} d\tau_0\ldots d\tau_{\Neps}
\mathcal{L}({\bm \tau},{\bm J},N_\varepsilon, u,\bm\theta)\\
& =\sum_{l^*_1, \dots, l^*_k}{}  \biggl\{
   \biggl[      
    \prod_{i=1}^{k}\int J_{l^*_i}^{n_i}  \delta_{\tau_{l^*_i}} (\theta^*_i)
e^{-uJ_{l^*_i}}\rho_\varepsilon(J_{l^*_i})P_0(\tau_{l^*_i})
    dJ_{l^*_i}d\tau_{l^*_i}   
    \biggr]   \\   
   &  \qquad \times\biggl[
    \prod_{j \neq \{ l^*_1, .., l^*_k \} } \int e^{-uJ_{j}}\rho_\varepsilon(J_{j})P_0(\tau_{j})
    dJ_{j}d\tau_{j} 
    \biggr]   
   \biggr\}
   \frac{1}{\Gamma(n)}u^{n-1} e^{-\Lambda_\varepsilon}\dfrac{\Lambda_\varepsilon^{N_\varepsilon}}{N_\varepsilon!}\\
   & =  \sum_{l^*_1, \dots, l^*_k}{}  \biggl\{\biggl[      
    \prod_{i=1}^{k}\int J_{l^*_i}^{n_i} \dfrac{\e^{-u J_{l^*_i}}e^{-uJ_{l^*_i}}\rho(J_{l^*_i})}{\Lambda_\varepsilon}dJ_{l^*_i}P_0({\theta_i}^*	)\biggr]
      \prod_{j \neq \{ l^*_1, .., l^*_k \}}\biggl(\dfrac{ \Lambda_{\varepsilon, u}}{\Lambda_\varepsilon}\biggr) \biggr\} 
       \frac{u^{n-1}}{\Gamma(n)} e^{-\Lambda_\varepsilon}\dfrac{\Lambda_\varepsilon^{N_\varepsilon}}{N_\varepsilon!}\\
   &=\frac{u^{n-1}}{\Gamma(n)}e^{-\Lambda_\varepsilon}\dfrac{\Lambda_\varepsilon^{N_\varepsilon}}{N_\varepsilon!}
    \sum_{l^*_1, \dots, l^*_k}{} \biggl\{ \dfrac{1}{\Lambda_\varepsilon^k} 
    \prod_{i=1}^k \biggl(P_0(\theta^*_i)\int_{\varepsilon}^{+\infty}\kappa s^{n_i}\e^{-u s}\rho(s)ds \biggr) \dfrac{\Lambda_{\varepsilon,u}^{N_\varepsilon+1-k}}{\Lambda_\varepsilon^{N_\varepsilon+1-k}} \biggr\}\\
    & = \frac{u^{n-1}}{\Gamma(n)}e^{-\Lambda_\varepsilon}\dfrac{\Lambda_{\varepsilon,u}^{N_\varepsilon+1-k}}{\Lambda_\varepsilon}
    \dfrac{(N_\varepsilon+1)}{(N_\varepsilon+1-k)!}\prod_{i=1}^k\biggl( P_0(\theta^*_i)\int_{\varepsilon}^{+\infty}\kappa s^{n_i}\e^{-u s}\rho(s)ds  \biggr).
\end{align*}

We have already introduced in this paper $N_{na}= N_\varepsilon+1-k$, the number of non-allocated jumps. 
Of course, the conditional distribution $ \calL(\Neps|u, \boldsymbol{\theta})$ in \eqref{eq:post_peps} is identified by $\mathcal{L}(N_{na}|u, \boldsymbol{\theta}) $, which can be derived as
\begin{equation}
\label{eq:fullcond_Nna}
\begin{split}
 \mathcal{L}(N_{na}|u, \boldsymbol{\theta}) & \propto \mathcal{L}(N_{na},u, \boldsymbol{\theta}) \propto 
 \Lambda_{\varepsilon,u}^{N_{na}} \dfrac{(N_{na}+k)}{N_{na}!} \\ 
 & \propto  e^{-\Lambda_{\varepsilon,u}}\biggl( \dfrac{\Lambda_{\varepsilon,u}}{(N_{na}-1)!}\Lambda_{\varepsilon,u}^{N_{na}-1} +
 \dfrac{k}{N_{na}!}\Lambda_{\varepsilon,u}^{N_{na}}\biggr)\uno_{(N_{na}\geqslant0)}\\
& \propto \dfrac{\Lambda_{\varepsilon,u}}{\Lambda_{\varepsilon,u} + k}\mathcal{P}_1(N_{na};\Lambda_{\varepsilon,u}) +  
 \dfrac{k}{\Lambda_{\varepsilon,u} + k}\mathcal{P}_0(N_{na};\Lambda_{\varepsilon,u}).
\end{split}
\end{equation}
On the other hand, the first factor in the right hand side of \eqref{eq:post_peps} can be computed 
by introducing ${\bm l}^*=(l^*_1,\ldots,l^*_{\correz{k}} )$, the vector of indexes of the allocated jumps
and by observing that the augmented right hand side of \eqref{eq:post_peps}
\begin{equation}
\begin{split}
&\mathcal{L}({\bm J},{\bm \tau}, {\bm l}^*|N_{na}, u,\bm\theta)  = 
 J_{l^*_1}^{n_1}
  \delta_{\tau^*_{l^*_1}}(\theta^*_1) \ldots  J_{l^*_k}^{n_k}
\delta_{\tau^*_{l^*_k}} (\theta^*_k) 
\prod_{j=0}^{N_{na}+k-1}\rho_\varepsilon(J_j) P_0(\tau_j)\e^{-u J_j}\\
& \quad= \left( \prod_{i=1}^k   J_{l^*_i}^{n_i} \e^{-u J_{l^*_i}} \dfrac{\kappa\rho(J_{l^*_i})}{\Lambda_\varepsilon}
\delta_{\tau_{l^*_i}}(\theta^*_i )P_0(\tau_{l^*_i})
\right) \left( \prod_{j \neq \{ l^*_1, .., l^*_k \} }  e^{-uJ_{j}}\dfrac{\kappa\rho(J_{j})}{\Lambda_\varepsilon}P_0(\tau_{j}) 
\right)\\
& \quad = \dfrac{1}{\Lambda_\varepsilon^{N_\varepsilon+1}}
\biggl( \prod_{i=1}^k   
J_{l_i^*}^{n_i}\e^{-u J_{l_i^*}}\kappa  \rho(J_{l_i^*})\delta_{\tau_{l_i^*}} P_0(\tau_{l^*_i})\biggr)
\left( \prod_{j \neq \{ l^*_1, .., l^*_k \} }^{} \e^{-u J_j} \kappa \rho(J_j)P_0(\tau_{j}) \right).
\label{eq:postAugm}
\end{split}
\end{equation}
The first factor in the last expression refers to the unnormalized \emph{allocated} process: the support is ${\bm \theta}^*$.
This shows point \ref{prop:alloc}. of the Proposition.

Therefore, the conditional distribution of $P_\varepsilon$ is proportional to the following expression: 
\begin{align*}
 \mathcal{L}(P_\varepsilon|u, \boldsymbol{\theta}) & \propto 
 \sum_{l^*_1, \dots, l^*_k} \dfrac{1}{\Lambda_\varepsilon^{N_\varepsilon+1}}
\biggl( \prod_{i=1}^k   
J_{l_i^*}^{n_i}\e^{-u J_{l_i^*}} \kappa \rho(J_{l_i^*})\delta_{\tau_{l_i^*}} P_0(\tau_{l^*_i})\biggr)
\left( \prod_{j \neq \{ l^*_1, .., l^*_k \} }^{} \e^{-u J_j} \kappa \rho(J_j)P_0(\tau_{j}) \right)\\
& \hspace{4cm} \times \dfrac{\Lambda_{\varepsilon,u}^{N_\varepsilon+1-k}}{\Lambda_\varepsilon}e^{-\Lambda_\varepsilon}
\dfrac{(N_\varepsilon+1)}{(N_\varepsilon+1-k)!}\prod_{i=1}^k P_0(\theta^*_i)\int_\varepsilon^{+\infty}s^{n_i}\kappa\e^{-u s}\rho(s)ds. 
\end{align*}
This yields points 1.,3. and 4. of the Proposition.

To show point \ref{prop:margU}., we need to integrate out $\Neps$ from $\LL(\Neps,u,{\bm\theta})$; we have: 
\begin{align*}
 \mathcal{L}(u|\boldsymbol{\theta}) &\propto \sum_{N_\varepsilon=0}^{+\infty} \mathcal{L}(N_\varepsilon,u, \boldsymbol{\theta}) = 
 \sum_{N_\varepsilon=0}^{+\infty}\dfrac{u^{n-1}}{\Gamma(n)}\e^{-\Lambda_\varepsilon}\dfrac{\Lambda_{\varepsilon,u}^{N_\varepsilon+1-k}}{\Lambda_\varepsilon}
 \dfrac{N_\varepsilon+1}{(N_\varepsilon+1-k)!}\prod_{i=1}^k\int_\varepsilon^{+\infty}\kappa s^{n_i}\e^{-u s}\rho(s)ds\\
 & = \dfrac{u^{n-1}}{\Gamma(n)}\e^{-\Lambda_\varepsilon}\sum_{N_{na}=0}^{+\infty}\dfrac{\Lambda_{\varepsilon,u}^{N_{na}}}{\Lambda_\varepsilon}
 \dfrac{N_{na}+k}{N_{na}!}\biggl(\prod_{i=1}^k\int_\varepsilon^{+\infty}\kappa s^{n_i}\e^{-u s}\rho(s)ds \biggr)\\
& = \dfrac{u^{n-1}}{\Gamma(n)}e^{\Lambda_{\varepsilon,u}-\Lambda_\varepsilon}\dfrac{\Lambda_{\varepsilon,u}+k}{\Lambda_\varepsilon}
\biggl(\prod_{i=1}^k\int_\varepsilon^{+\infty}\kappa s^{n_i}\e^{-u s}\rho(s)ds \biggr).
 \end{align*}
This ends the proof.
\end{proof}

\section{Prior moments of $\Peps$} 
\label{sec:moments}
Before deriving the first two moments of $\Peps$, let us mention that the expect value and variance of $N_\varepsilon$, the number of jumps considered in the approximation $\Peps$, depend on the prior of $\varepsilon$. Of course, if $\varepsilon$ is assumed fixed, $ \E(N_\varepsilon)= \Var(N_\varepsilon)=\Lambda_\varepsilon<+\infty$, while, if  $\varepsilon$ is random, then   
\begin{equation*}
 \E(N_\varepsilon)=\E((N_\varepsilon|\varepsilon))=\E(\Lambda_\varepsilon), \quad  \Var(N_\varepsilon)=\Var(\Lambda_\varepsilon) + \E(\Lambda_\varepsilon). 
\end{equation*}
In this case, the mean and variance of $N_\varepsilon$ are not necessarily finite; see, for instance, Table 2 in
\cite{peps},  where $\Peps$ is the $\varepsilon$-NGG process, and for some values of its hyperparameters the mean or the variance of $N_\varepsilon$ are infinite. 

First of all, observe that 
\begin{align}
\left(x_1+\cdots+x_{\Neps^*} \right)^m & = 
\sum_{\substack{ m_1+\cdots+m_{\Neps^*}=m \label{eq:multinomial}\\
m_1,\ldots,m_{\Neps^*} \ge 0}}
 \binom{m}{m_1,\ldots,m_ {\Neps^*}}  \prod_{j=1}^{\Neps^*} x_j^{m_j}\\
& = \sum_{k=1}^m \uno_{\{1,\ldots, \Neps^* \}} (k) \frac{1}{k!} 
\sum_{\substack{n_1+\cdots+n_k=m\\ n_j=1,2,\ldots}}  \binom{m}{n_1,\ldots,n_k}  \left( \sum_{j_1,\ldots,j_k}\prod_{i=1}^k x_{j_i}^{n_i} \right)
\end{align}
where $\Neps^*=\Neps+1$, $x_j^0=1$ for all $x_j\geq 0$, and the last summation is over all positive integers, being \eqref{eq:multinomial} the multinomial theorem. The second equality follows straightforward from different identifications of the set of all partitions of $m$ \citep[see][Section 1.2]{Pitman06}. Therefore, for any $B\in\calB(\Theta)$, $m=1,2,\ldots$,  we have 
(here, instead of $P_0$ and $\tau_0$ as in \eqref{eq:Pvarepsilon}, there are $P_{\Neps^*}$ and $\tau_{\Neps^*}$):
\begin{align*}
   \E&(\Peps(B)^m )  = \E\left( \E \left( (\sum_{j=1}^{\Neps^*}
P_j\delta_{\tau_j}(B))^m |\Neps \right)\right)\\ 
&= \E\left( \E \left( \sum_{\substack{ m_1+\cdots+m_{\Neps^*}=m\\
m_1,\ldots,m_{\Neps^*} \ge 0}}
 \binom{m}{m_1,\ldots,m_ {\Neps^*}}  \prod_{j=1}^{\Neps^*} (P_j\delta_{\tau_j}(B))^{m_j} |\Neps \right) \right)\\
&=  \E\left( \E \left(  \sum_{k=1}^m \uno_{\{1,\ldots, \Neps^* \}} (k) \frac{1}{k!} 
\sum_{\substack{n_1+\cdots+n_k=m\\ n_j=1,2,\ldots}}  \binom{m}{n_1,\ldots,n_k}  \left( \sum_{j_1,\ldots,j_k}\prod_{i=1}^k (P_{j_i}\delta_{\tau_{j_i}}(B))^{n_i} \right)   |\Neps \right) \right)\\
 &=  \E\left(  \sum_{k=1}^m \uno_{\{1,\ldots, \Neps^* \}} (k) \frac{1}{k!} \sum_{\substack{n_1+\cdots+n_k=m\\ n_j=1,2,\ldots}}  \binom{m}{n_1,\ldots,n_k}  
 \sum_{j_1,\ldots,j_k} \E(\prod_{i=1}^k P_{j_i} ^{n_i}|\Neps) \prod_{i=1}^k \E(\delta_{\tau_j}(B)|\Neps)
 \right)\\
  \end{align*}
 \begin{equation*}
= \E\left(  \sum_{k=1}^m \uno_{\{1,\ldots, \Neps^* \}} (k) \frac{1}{k!} \sum_{\substack{n_1+\cdots+n_k=m\\ n_j=1,2,\ldots}}  \binom{m}{n_1,\ldots,n_k}  p_{\varepsilon}(n_1,\ldots,n_k) (P_0(B))^k\right).
\end{equation*}
We identify this last expression as 
\begin{equation*}
 \E\left(  \sum_{k=1}^m P_0(B)^k \Prob(K_m=k|\Neps)  \right),
\end{equation*}
where $K_m$ is the number of distinct values in a sample of size $m$ from $\Peps$.
Hence, we have proved that 
\begin{align*}
\E(\Peps(B)^m )  &= \E \left( \E( P_0(B)^{K_m} |\Neps)\right) = \E \left( P_0(B) ^{K_m}\right). 
\end{align*}
In particular, when $m=2$, $K_m$ assumes value in $\{1,2 \}$, and the probability that $K_2=1$ is the probability that, in a sample of size 2 from $\Peps$, the samples values coincide, i.e. $p_\varepsilon(2)$. 
Therefore 
$$ \E(\Peps(B)^2) = P_0(B)p_\varepsilon(2) + (P_0(B))^2(1-p_\varepsilon(2)), $$
and consequently 
\begin{align}
\label{eq:var_peps}
\Var(\Peps(B))= P_0(B)p_\varepsilon(2) + P_0(B)^2(1-p_\varepsilon(2)) +  P_0(B)^2
= p_\varepsilon(2) P_0(B) \left( 1-P_0(B)\right). 
\end{align}   
Analogously, suppose that  $B_1,B_2\in\calB(\Theta)$ are disjoint. Therefore
\begin{align*}
\E(\Peps(B_1)\Peps(B_2)) &= \E\left( \E\left (\sum_{j=1}^{\Neps^*} P_j\delta_{\tau_j}(B_1) \sum_{l=1}^{\Neps^*} P_l\delta_{\tau_l}(B_2) |\Neps\right) \right) \\
&=   \E\left( \E \left(\sum_{j=1}  ^{\Neps^*} P_j^2 \delta_{\tau_j}(B_1\cap B_2)  + \sum_{\substack{l\neq j\\ j,l=1,\ldots,\Neps^*}} P_j P_l \delta_{\tau_j}(B_1)\delta_{\tau_l}(B_2)
 |\Neps\right) \right)\\
&=    \E\left( \sum_{\substack{l\neq j\\ j,l=1,\ldots,\Neps^*}} \E(P_j P_l|\Neps)  \E(\delta_{\tau_j}(B_1))\E(\delta_{\tau_l}(B_2)))  \right) \\
&= \E\left( P_0(B_1) P_0(B_2)  \sum_{\substack{l\neq j\\ j,l=1,\ldots,\Neps^*}} \E(P_j P_l|\Neps) \right) 
=  P_0(B_1) P_0(B_2) p_\varepsilon(1,1).
\end{align*}
The general case when $B_1$ and $B_2$ are not disjoint follows easily: 
\begin{align*}
\E(\Peps(B_1)\Peps(B_2)) &= \E\left( (\Peps(B_1\cap B_2))^2 \right)+ 
\E \left( \Peps(B_1\setminus B_2) \Peps(B_1\cap B_2) \right) \\
&\quad+ \E \left( \Peps(B_2\setminus B_1) \Peps(B_1\cap B_2) \right) +\E(\Peps(B_1\setminus B_2) \Peps(B_2\setminus B_1) ),
\end{align*} 
where now the sets are disjoint. Applying the result above we first find that
\begin{equation*}
\E(\Peps(B_1)\Peps(B_2)) = p_\varepsilon(2) P_0(B_1\cap B_2)+ (1- p_\varepsilon(2)) P_0(B_1) P_0(B_2),
\end{equation*}
and consequently:
\begin{equation*}
\Cov((\Peps(B_1),\Peps(B_2)) = p_\varepsilon(2)  \left(  P_0(B_1\cap B_2) - P_0(B_1) P_0(B_2) \right). 
\end{equation*}


\section{$\varepsilon$-NormCRM process mixtures}
\label{sec:epsHNRMImixtures}
Among the wide range of applications in which discrete random probability measures are exploited,
hierarchical mixture models, dating back to \cite{Lo84}, are frequently used when dealing with various data structures.
Hence, as argued in the Introduction, their role is becoming more and more central in modern Bayesian Nonparametrics. 
We consider mixtures of parametric kernels as the distribution of data, where the mixing measure is the 
$\varepsilon$-NormCRM$(\rho, \kappa P_0)$. 
The model we assume is the following: 
\begin{equation}
\begin{split}
  \label{eq:modello}
  Y_i|\theta_i  &\ind f(\cdot;\theta_i),  \  i=1,\ldots, n \\
  \theta_i|\Peps  &\iid P_{\varepsilon}, \  i=1,\ldots, n \\
 P_\varepsilon &\sim \varepsilon-NormCRM(\rho, \kappa P_0),  \\
\varepsilon &\sim \pi(\varepsilon), 
\end{split}
\end{equation} 
where $f(\cdot;\theta_i)$ is a parametric family of densities on $\mathbb{Y}\subset\Rea^p$, for all $\theta\in\Theta \subset \Rea^m$.  
Remember that
$P_0$ is a non-atomic probability measure on $\Theta$, such that 
$\E(\Peps(A))=P_0(A)$ for all $A\in\calB(\Theta)$ and all $\varepsilon\geq 0$.
Model \eqref{eq:modello} will be addressed here as 
$\varepsilon-$\textit{NormCRM hierarchical mixture} model. It is well known that this model is equivalent to assume that the $Y_i$'s, conditionally on $P_\varepsilon$, are independently distributed according to the random density 
\begin{equation*}
h(y)=\int_\Theta f(y;\theta)P_\varepsilon(d \theta) = \sum_{j=0}^{N_\varepsilon} P_j \  f(y;\tau_j).
\end{equation*}
In particular, we are able to build a blocked Gibbs sampler to
update blocks of parameters, which are drawn from  multivariate distributions.

The parameter is $(\Peps,\varepsilon,{\bm\theta})$, 
but we use the augmentation trick prescribed by the posterior characterization in Proposition \ref{prop:Posterior}, 
so that the new parameter is  $(\Peps,\varepsilon,{\bm\theta}, u)$; 
the joint law of data and parameters can be written as follows: 
\begin{equation}\label{eq:leggecongiuntamodelloAPP}
 \begin{split}
\mathcal{L}(\boldsymbol{Y}, \boldsymbol{\theta}, u, P_\varepsilon, \varepsilon)
&=\mathcal{L}(\boldsymbol{Y}| \boldsymbol{\theta}, u, P_\varepsilon, \varepsilon)
\mathcal{L}(\boldsymbol{\theta}, u, P_\varepsilon| \varepsilon)
\mathcal{L}(\varepsilon)
=\prod_{i=1}^{n}f(Y_i; \theta_i)\mathcal{L}(\boldsymbol{\theta}, u, P_\varepsilon| \varepsilon)
\pi(\varepsilon)\\
&=\frac{u^{n-1}}{\Gamma(n)} \prod_{j=0}^{N_\varepsilon} 
\big( e^{-u J_j}\rho_\varepsilon(J_j)P_0(\tau_j) \big)\sum_{l^*_1, .., l^*_k}^{} \big(  J_{l^*_1}^{n_1}
   \prod_{i \in C_1}^{}f(Y_i;\theta^*_1)\delta_{\tau_{l^*_1}}(\theta^*_1)\times\\
&   \qquad\qquad \qquad\qquad \cdots\times J_{l^*_k}^{n_k}\prod_{i \in C_k}^{}f(Y_i;\theta^*_k) \delta_{\tau_{l^*_k}}(\theta^*_k) \big)
  \dfrac{\Lambda_\varepsilon^{N_\varepsilon}e^{-\Lambda_\varepsilon}}{N_\varepsilon!}\pi(\varepsilon),\\
  \end{split}
\end{equation}
where we used the hierarchical structure in $\eqref{eq:modello}$. 
The Gibbs sampler generalizes that one provided in \cite{peps} for $\varepsilon$-NGG mixtures.  
Description of the  full-conditionals is below, and further details can be found in the Appendix. 
\begin{enumerate}
\item[1.] \textbf{Sampling from $\boldsymbol\LL(u|{\bm Y},{\bm\theta},\Peps,\varepsilon)$}: 
from $\eqref{eq:leggecongiuntamodelloAPP}$ it is easy to see that the factors depending on $u$ identify this full-conditional as gamma with parameters $(n,T_\varepsilon)$,  like the corresponding prior. 
\item[2.] \textbf{Sampling from 
  $\boldsymbol\LL({\bm\theta}|u, {\bm Y},\Peps,\varepsilon)$}: 
each $\theta_i$, for $i=1,\dots,n$, has  discrete law with support 
  $\{\tau_0,\tau_1,\dots,\tau_{\Neps} \}$, and probabilities 
  $\mathbb{P}(\theta_i=\tau_j) \propto J_j f(Y_i; \tau_j)$.
\item[3.] \textbf{Sampling from $
\boldsymbol\LL(\Peps,\varepsilon|u,{\bm\theta}, {\bm Y})$}: 
this step is not 
straightforward and can be split into two consecutive substeps:
\begin{enumerate}
  \item[3.a] \textbf{Sampling from $\boldsymbol{\mathcal{L}}(\varepsilon|u,
    \boldsymbol\theta, \boldsymbol{Y})$}: see the Appendix.  
   \item[3.b] \textbf{Sampling from  $\boldsymbol{\mathcal{L}}(P_{\varepsilon}|\varepsilon, u, \boldsymbol\theta,
\boldsymbol{Y})$}: via characterization of 
 the posterior in Proposition~\ref{prop:Posterior}, since this distribution is equal to
 $\mathcal{L}(P_{\varepsilon}|\varepsilon,u, \boldsymbol\theta)$.
 To put into practice, we have to sample  $(i)$ the number $N_{{na}}$ of 
 \emph{non-allocated} jumps, $(ii)$  the vector 
 of the unnormalized  \emph{non-allocated} jumps $\bm{J}^{(na)}$, 
 $(iii)$  the vector 
 of the unnormalized \emph{allocated} jumps $\bm{J}^{(a)}$,
  the support of the \emph{allocated} $(iv)$ and 
\emph{non-allocated} $(v)$ jumps. See the Appendix for a wider description. 
\end{enumerate}
\end{enumerate}
Remember that, when sampling from non-standard
distributions, Accept-Reject or Metropolis-Hastings
algorithms have been exploited.

\section{Normalized Bessel random measure mixtures: an application to density estimation}
\label{sec:bessel}

In this section 
we introduce a new normalized process, called normalized Bessel random measure, 
corresponding to a specific choice for 
the intensity function $\rho(\cdot)$. Section~\ref{sec:bessel_def} describes theoretical results: in particular, 
we show that this family encompasses the well-known Dirichlet process.
Then we fit the mixture model to  synthetic and 
real datasets in Section~\ref{sec:Bessel_application}. Results are illustrated through a density estimation problem.

\subsection{Definition}\label{sec:bessel_def}
Let us consider a  normalized completely random measure corresponding to mean intensity 
\begin{equation}
\rho(s; \omega)= \frac{1}{s}\e^{-\omega s} I_0(s), \quad s>0,
\end{equation}
where $\omega\geq 1$ and 
\begin{equation}
I_\nu(s)= \sum_{m=0}^{+\infty}\frac{(s/2)^{2m+\nu}}{m!\Gamma(\nu+m+1)}
\end{equation}
is the modified Bessel function of order $\nu>0$ \citep[see][Sect 7.2.2]{erdelyi53}. It is straightforward to see that, for $s>0$, 
\begin{equation}
\label{eq:2super}
\rho(s;\omega)=  \frac{1}{s}\e^{-\omega s} + \sum_{m=1}^{+\infty}\frac{1}{2^{2m}(m!)^2}s^{2m-1} \e^{-\omega s},
\end{equation}
so that $\rho$ is the sum of the L\'evy intensity of the gamma process with rate parameter $\omega$  and of the L\'evy intensities 
\begin{equation}
\label{eq:bessel_rhom}
\rho_m(s;\omega)= \frac{1}{2^{2m}(m!)^2}s^{2m-1} \e^{-\omega s}, \quad s>0, \qquad m=1,2,\ldots
\end{equation}
corresponding to  finite activity Poisson processes. It is simple to check that 
\eqref{eq:regolarita} holds. Hence, following  \eqref{eq:NRMI} in  Section~\ref{sec:homogen}, we introduce the \emph{normalized Bessel random measure} $P$, with parameters $(\omega,\kappa)$,  where $\omega\geq 1$ and $\kappa>0$. 
Thanks to \eqref{eq:2super} and the Superposition Property of Poisson processes, in this case, the total mass $T$ in \eqref{eq:NRMI}  can be written as 
\begin{equation}
\label{eq:bessel_decomp}
 T\stackrel{d}= T_G +\sum_{m=1}^{+\infty}T_m,
\end{equation}
where $ T_G$, $T_1$, $T_2,\ldots$ are independent random variables, $T_G$ being the total mass of the gamma process and $T_m$ the total mass of a completely random measure corresponding to the  intensity 
$\nu_m(ds,d\tau)=\rho_m(s)ds \kappa P_0(d\tau)$. In particular, $T_G\sim gamma(\kappa,\omega)$, while $T_m=\sum_ {j=0}^{N_m} J_j^{(m)}$, where $N_m\sim Poi	(\kappa \Gamma(2m)/( (2\omega)^{2m} (m!)^2))$, 
 and $ \{ J_j^{(m)} \}$ are the points of a Poisson process on $\Rea^+$ with intensity $\kappa \rho_m$. By this notation we mean that $T_m$ is equal to 0 when $N_m=0$, while, conditionally to $N_m>0$,
 $J_j^{(m)}\iid gamma(2m,\omega)$.  
We can write down the density function of $T$, via \eqref{eq:LaplaceTransf}:
\begin{align*}
\psi(\lambda) &:= -\log \left(\E(\e^{-\lambda T})\right) = \kappa \int_0^{+\infty}(1-\e^{-\lambda s}) \rho(s;\omega)ds\\
& = \kappa \left(  \int_0^{+\infty} (1-\e^{-\lambda s})  \frac{1}{s}\e^{-\omega s}ds + \sum_{m=1}^{+\infty}\frac{1}{2^{2m}(m!)^2} \int_0^{+\infty}  (1-\e^{-\lambda s})  
s^{2m-1}\e^{-\omega s}ds \right)\\
&= \kappa \left( \log \left( \frac{\omega+\lambda}{\omega}\right) +\sum_{m=1}^{+\infty}\frac{\Gamma(2m)}{2^{2m}(m!)^2 \omega^m} - \sum_{m=1}^{+\infty}\frac{\Gamma(2m)}{2^{2m}(m!)^2 (\omega+\lambda)^m}
\right)\\
&=  \kappa \left( \log \left( \frac{\omega+\lambda}{\omega}\right) - \log \left( \frac12+\frac12\sqrt{1-\frac{1}{\omega^2}}\right) +   \log \left( \frac12+\frac12\sqrt{1-\frac{1}{(\omega+\lambda)^2}}\right) \right) \\
&= \kappa \log \left( \frac{\omega+\lambda+\sqrt{(\omega+\lambda)^2-1}}{\omega+\sqrt{\omega^2-1}} \right).
\end{align*}
The same expression is obtained when $T\sim f_T(t)=\kappa (\omega+\sqrt{\omega^2-1})^\kappa 
\dfrac{\e^{-\omega t}}{t}I_{\kappa}(t)$, $t>0$ \citep[see][formula (17.13.112)]{GraRyz07}.  Observe that, when $\omega=1$, $f_T$ is called \textit{Bessel function density}
\citep{feller1971}. By \eqref{eq:eppf_hnrmi}, the eppf of the normalized Bessel random measure is: 
\begin{equation}
\label{eq:eppf_Bessel}
\begin{split}
p_B(n_1,\ldots,n_k; \omega,\kappa) &= \kappa^k \int_0^{+\infty} \frac{u^{n-1}}{\Gamma(n)} \left( \frac{\omega+\sqrt{\omega^2-1}} {\omega+u+\sqrt{(\omega+u)^2-1}} \right)^\kappa
\frac{1}{(u+\omega)^n}\\
&\qquad\qquad\qquad \times \prod_{j=1}^k\Gamma(n_j) \   \ff \left(\frac{n_j}{2}, \frac{n_j+1}{2};1; \frac{1}{(u+\omega)^2}\right) du,  
\end{split}
\end{equation}
where
$$ \ff(\alpha_1,\alpha_2;\gamma;z) :=  \sum_{m=0}^{\infty}
\frac{  \left( \alpha_1\right)_m
\left( \alpha_2 \right)_m
}{\left( \gamma\right)_m}
\frac{1}{m!} \left( z\right)^m, \quad\textrm{ with }  (\alpha)_m :=\frac{\Gamma(\alpha+m)}{\Gamma(\alpha)} $$ is the hypergeometric series \citep[see][formula (9.100)]{GraRyz07}.

The following proposition shows that the eppf of the normalized Bessel random measure converges to the eppf of the Dirichlet process as the parameter $\omega$ increases. 
\begin{proposition}
Let  $(n_1,\dots,n_k)$ be a vector of positive integers such that $\sum_{i=1}^kn_i=n$, where $k=1,\ldots,n$. Then, the eppf  \eqref{eq:eppf_Bessel}, associated with the normalized Bessel random measure $P$ with parameter $(\omega,\kappa)$, $\omega\geq 1$, $\kappa>0$,  and mean measure $P_0$, is such that 
\begin{equation*}
\lim_{\omega\rightarrow +\infty}p_B(n_1,\ldots,n_k;\omega,\kappa) = p_{D}(n_1,\dots,n_k;\kappa), 
\end{equation*}
where $p_{D}(n_1,\dots,n_k;\kappa)$ is  the eppf of the Dirichlet process with measure parameter $\kappa P_0$.
\end{proposition}

\begin{proof}
The eppf of the Dirichlet process appeared first in \cite{Antoniak74} \citep[see][]{Pitman96};
anyhow, it is straightforward to derive it from \eqref{eq:eppf_hnrmi}:
\begin{equation}
\label{eq:eppf_DP}
\begin{split}
&p_{D}(n_1,\dots,n_k;\kappa) =   \int_{0}^{+\infty} \frac{u^{n-1}}{\Gamma(n)} \e^{-\kappa \log\frac{u+\omega}{\omega}}  \prod_{j=1}^{k}\kappa\frac{\Gamma(n_j)}{(u+\omega)^{n_j}}du\\
&\quad\quad =\kappa^k \int_{0}^{+\infty} \frac{u^{n-1}}{\Gamma(n)} \left( \frac{\omega}{\omega+u} \right)^\kappa \frac{1}{(u+\omega)^n} \prod_{j=1}^{k}\Gamma(n_j)du  
 = \frac{\Gamma(\kappa)}{\Gamma(\kappa+n)}\kappa^k \prod_{j=1}^k\Gamma(n_j)
\end{split}
\end{equation}
where the last equality follows from formula (3.194.3) in  \cite{GraRyz07}.
By  definition of the hypergeometric function, we have
$$1 \le \ff\left( \frac{n_j}{2},\frac{n_j+1}{2};1;\frac{1}{(u+\omega)^2} \right)\le
\ff\left( \frac{n_j}{2},\frac{n_j+1}{2};1;\frac{1}{\omega^2} \right) \ . $$ 
Moreover 
\begin{equation*} 
  \frac{\omega+\sqrt{\omega^2-1}}{  (u+\omega)+\sqrt((u+\omega)^2-1)} =
  \frac{\omega}{u+\omega}\frac{1+\sqrt{1-1/\omega^2}}{1+\sqrt{1-1/(u+\omega)^2}}
\end{equation*}
and 
\begin{equation*}
\frac{1+\sqrt{1-1/\omega^2}}{2}  \le 
\frac{1+\sqrt{1-1/\omega^2}}{1+\sqrt{1-1/(u+\omega)^2}} \le 1,
\end{equation*}
so that 
\begin{align*}
  \left( \frac{1+\sqrt{1-1/\omega^2}}{2} \right)^\kappa p_{D}(n_1,\dots,n_k;\kappa)
 & \le p_{B}(n_1,\dots,n_k;\omega,\kappa)\\
&\le
  \prod_{j=1}^{k}\ff
\left( \frac{n_j}{2},\frac{n_j+1}{2};1;\frac{1}{\omega^2}\right)
p_{D}(n_1,\dots,n_k;\kappa). 
\label{eq:disuguaglianza}
\end{align*}
The left hand-side of these inequalities obviously converges to $p_{D}(n_1,\dots,n_k;\kappa)$ as $\omega$ goes to $+\infty$. On the other hand, 
$$\ff\left( \frac{n_j}{2},\frac{n_j+1}{2};1;\frac{1}{\omega^2}\right)\rt 1 \textrm{ as } \omega\rt +\infty,$$
thanks to the uniform convergence of the hypergeometric series $\ff( \frac{n_j}{2},\frac{n_j+1}{2};1;z)$ on a disk of radius smaller that 1.
We conclude that, for any $n_1,\ldots,n_k$ such that $n_1+\cdots+n_k=n$, $k=1,\ldots,n$, and any $\kappa>0$,
$$\lim_{\omega\rightarrow +\infty}p_B(n_1,\ldots,n_k;\omega,\kappa) = p_{D}(n_1,\dots,n_k;\kappa).$$ 
\end{proof}

Since the eppf is the joint distribution of the number $K_n$ of distinct values and corresponding sizes $N_1$,\dots,$N_k$ (see equation (30) in \cite{Pitman96})
in a sample of size $n$ from the normalized Bessel completely random measure, 
by marginalization we obtain 
\begin{equation*}
  P(K_n=k)=\frac{1}{k!} \sum_{n_1,\dots,n_k}^{} {n \choose{n_1,\dots,n_k}}
  p_B(n_1,\dots,n_k), \quad k=1,\ldots,n, 
\end{equation*}
where the sum is over all the compositions of $n$ into $k$ part, i.e., all positive integers such that $n_1+\dots+n_k=n$.
Unfortunately,  we were not able to simplify further this last expression,  because
of the summation of the hypergeometric functions $\ff$ occurring in the analytic expression  \eqref{eq:eppf_Bessel} of $p_B$. 
Since the number of partitions of $n$ items in $k$ blocks can be very high (it is given by the Stirling number $S(n;k)$ of the second kind) 
and the evaluation of $\ff$ computationally heavy, we prefer to use a Monte Carlo strategy to simulate from the prior of $K_n$. 
The simulation strategy is also useful to understanding the meaning of the parameters of the normalized Bessel random measure: $\kappa$ has the usual interpretation of the mass parameter, since,  when fixing $\omega$,
$\mathbb{E}(K_n)$ increases with $\kappa$.
On the other hand, the effect of $\omega$ is quite peculiar: 
decreasing $\omega$ (thus drifting apart from the Dirichlet process), with $\kappa$ fixed, the prior distribution of $K_n$ shifts towards smaller values. 
However, when $\mathbb{E}(K_n)$ is kept fixed, the distribution has heavier tails if $\omega$ is small (see Figures 
\ref{fig:prior_NB} and \ref{fig:simulated} (a)).

\begin{figure}[h!]
\centering {\includegraphics[width=0.5\textwidth,height=0.5\textwidth]{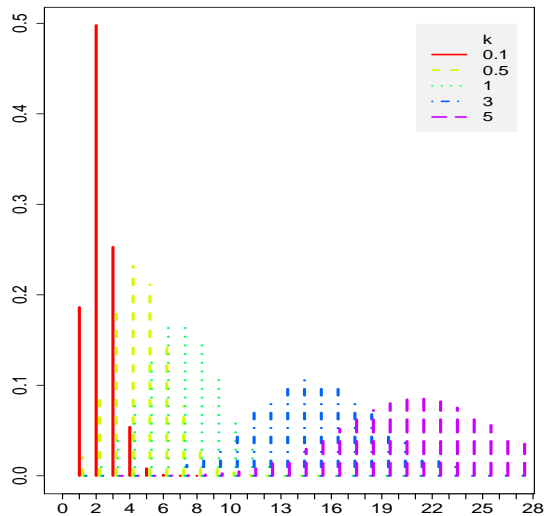}} 
\caption{Prior distribution of $K_n$ under a sample from $\varepsilon$-NB process with  $\varepsilon=10^{-6}$, 
$\omega=1.05$ and several values for $\kappa$, as reported in the legend.} 
\label{fig:prior_NB}
\end{figure}

\subsection{Application}
\label{sec:Bessel_application}

In this section let us consider the hierarchical mixture model \eqref{eq:modello}, where the mixing measure is $\Peps$, 
the $\varepsilon$-approximation of the normalized Bessel random measure, as
introduced above (here $\varepsilon$-NB($\omega$, $\kappa P_0$) mixture
model).  Of course, when $\varepsilon$ is small, this model approximates the
corresponding mixture when the mixing measure is $P$; to the best of our
knowledge, this normalized Bessel completely random measure has never been considered in the Bayesian nonparametric literature. 
By  decomposition  \eqref{eq:bessel_decomp},
we argue that this model is suitable when the unknown density shows many different components, where a few of them are very spiky (they should correspond to Levy intensities \eqref{eq:bessel_rhom}),  while  there is a folk of flatter components which are explained by the intensity $(1/s)\e^{-\omega s}$ of the Gamma process.
For this reason, we consider a simulated dataset 
 which is a sample from a mixture of 5 Gaussian distributions with means and standard deviations equal to 
$\{(15,1.1), (50,1), (20,4), (30,5), (40, 5) \}$, and weights proportional to $\left\{10, 9, 4, 5, 5\right\}$.
The histogram of the simulated data, for $n = 1000$, is reported in Figure~\ref{fig:sim_dens_est}.

We report posterior estimates for different sets of hyperparameters of the $\varepsilon$-NB mixture model when 
$f(\cdot; \theta)$ is the Gaussian density on $\mathbb{R}$ and $\theta = (\mu, \sigma^2)$ stands for its mean and variance.
Moreover, $P_0(d\mu, d\sigma^2)=\mathcal{N}(d\mu; \bar{y}_n, \sigma^2/\kappa_0)\times inv-gamma(d\sigma^2; a, b)$;
here $\mathcal{N}(\bar{y}_n, \sigma^2/\kappa_0)$ is the Gaussian distribution with mean $\bar{y}_n$(the empirical mean) 
and variance $\sigma^2/\kappa_0$, and $inv-gamma(d\sigma^2; a, b)$ is the inverse-gamma distribution with mean 
$b/(a-1)$ (if $a>1$). We set $\kappa_0=0.01$, $a=2$ and $b=1$ as proposed first in \cite{EscWes95}.
We shed light on three sets of hyperparameters in order to understand sensitivity of the estimates under different conditions of variability; indeed, each set has a different value of $p_\varepsilon(2)$, which tunes the a-priori variance of $\Peps$, as reported in 
\eqref{eq:var_peps}. We tested three different values for $p_\varepsilon(2)$:
$p_\varepsilon(2)=0.9$ in  set $(A)$, $p_\varepsilon(2)=0.5$ in set $(B)$ and  $p_\varepsilon(2)=0.1$  in set $(C)$. 
Moreover, in each scenario we let the parameter $1/\omega$ ranges in $\{0.01, 0.25, 0.5, 0.75, 0.95  \}$; note that the extreme case of 
$\omega=100$ (or equivalently $1/\omega=0.01$)  corresponds to an approximation of the  DPM model.
The mass parameter $\kappa$ is then fixed to achieve the desired level  of $p_\varepsilon(2)$. 
As far as the choice of $\varepsilon$ concerns, we set it equal to $10^{-6}$: 
at the end, we got 15 tests, listed in Table~\ref{tab:indexes}.
It is worth mentioning that it is possible to 
choose
a prior for $\varepsilon$,
even if, for the $\rho$ in \eqref{eq:2super}, 
the computational cost would greatly increase due to the evaluation of functions $\ff$ in \eqref{eq:eppf_Bessel}.

We have implemented our Gibbs sampler in C++.
All the tests in Sections~\ref{sec:bessel} and \ref{sec:LDNGG} were made on a laptop with Intel Core i7 2670QM processor, with 6GB of RAM. 
Every run produced a final sample size of 5000 iterations, after a thinning of 10 and an initial burn-in of 5000 iterations.
Every time the convergence was checked by standard R package CODA tools.
\begin{figure}[h!]
\centering 
\includegraphics[width=0.73\textwidth,height=0.62\textwidth]{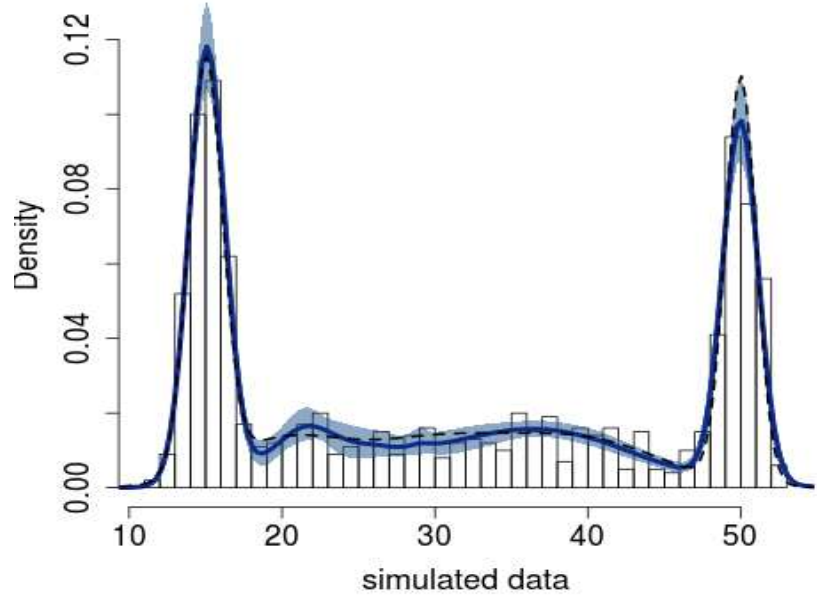}
\caption{Density estimate for case A5: posterior mean (line),  $90\%$
pointwise credibility intervals (shadowed area), true density (dashed) and the histogram of simulated data.}
 \label{fig:sim_dens_est}
\end{figure}
Here, we focus on density estimation: all the tests provide similar estimates, quite faithful to the true density.
\begin{figure}[h!]
\centering 
\subfigure[]%
{\includegraphics[width=0.45\textwidth,height=0.43\textwidth]{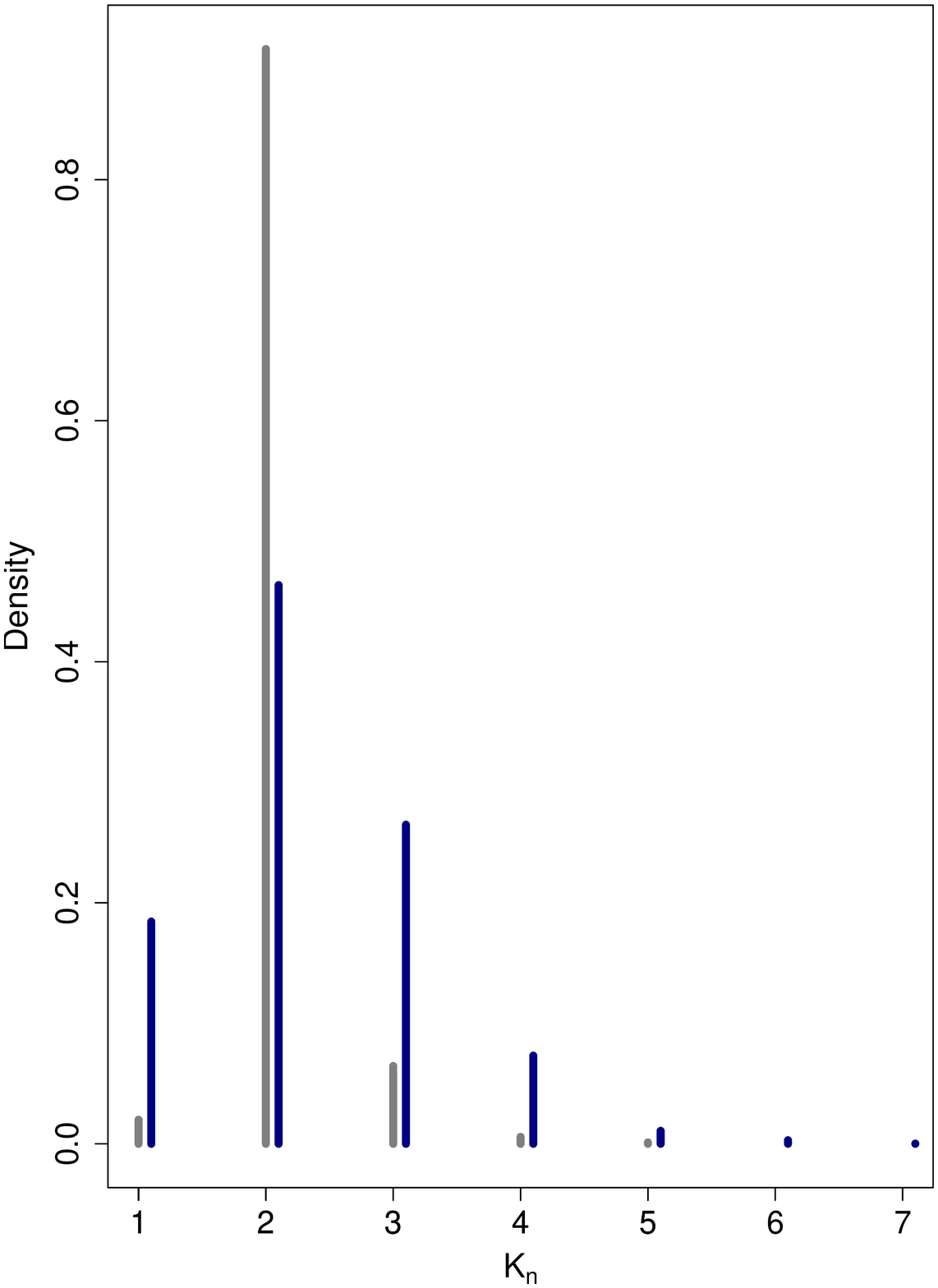}} 
\subfigure[]%
{\includegraphics[width=0.45\textwidth,height=0.43\textwidth]{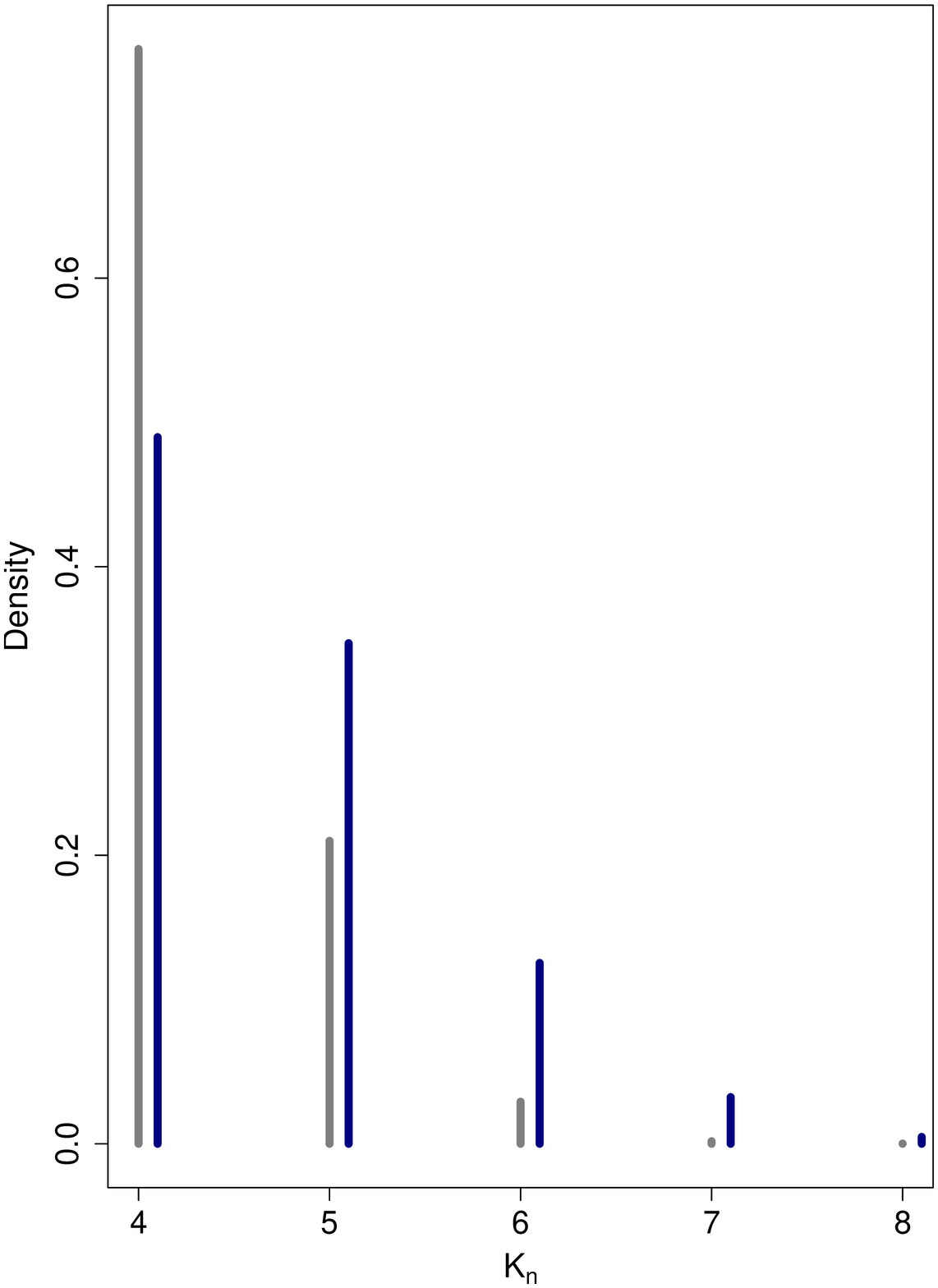}} 
\caption{Prior (a) and posterior (b) distributions of the number $K_n$ of groups for test A1 (gray) and A5 (blue).} 
\label{fig:simulated}
\end{figure}
Figure~\ref{fig:sim_dens_est} shows  density estimate  and  pointwise $90\%$ credibility intervals for case A5; the true density is  superimposed as dashed line. 
Figure~\ref{fig:simulated} (a) and (b) display prior and posterior distributions, respectively,  of the number 
$K_n$ of groups, i.e. the number of unique values among $(\theta_1,\ldots,\theta_n)$ in \eqref{eq:modello} under two sets of hyperparameters, 
A1, representing an approximation of the  DPM model, and  A5, where the parameter $\omega$ is nearly 1. 
From Figure~\ref{fig:simulated} it is clear that A5 is more flexible than A1: for case A5, a priori the variance of $K_n$ is larger, and, on the other hand, the posterior probability mass in 5 (the true value) is larger. 

In order to compare different priors, 
we take into account  five different predictive good\-ness-of-fit indexes:  
$(i)$ the sum of  squared errors (SSE) , i.e. the sum of the squared differences
between the $y_i$ and the predictive mean $\E(Y_i|data)$ (yes, we are using data twice!);
$(ii)$ the sum of standardized absolute errors  (SSAE), given by the sum of the standardized error $|y_i - \E(Y_i|data)|/\sqrt{\Var(Y_i|data)}$;  
$(iii)$ log-pseudo marginal likelihood (LPML), quite standard in the Bayesian literature, defined as 
the sum of $\log(CPO_i)$, where $CPO_i$ is the conditional predictive ordinate of $y_i$, the value of the predictive distribution evaluated at $y_i$, conditioning on the training sample given by all data except $y_i$. The last two indexes, $(iv)$ $WAIC_1$  and $(v)$ $WAIC_2$, as denoted here, were proposed in  \cite{Wat10} and deeply analyzed in \cite{GelHwa14}: they are  generalizations of the 
      AIC, adding two types of penalization, both accounting for the ``effective number of parameters''. 
The bias correction in ${WAIC_1}$ is similar to the bias correction in the definition of the DIC, while   ${WAIC_2}$ is the sum of the posterior variances of the conditional density of the data. See \cite{GelHwa14} for their precise definition.  
Table~\ref{tab:indexes} shows the values of the five indexes for each test: the optimal  (according to each index)  tests  are highlighted in bold 
for the experiments $(A)$, $(B)$ and $(C)$.
It is apparent that  the different tests provide similar values of the indexes, but SSE, indicating that, from a predictive viewpoint, 
there are no significant differences among the priors. 
However, especially when the value of $\kappa$ is small, i.e. in all tests $A$ and $B$, a model with a smaller $\omega$ tends to outperform the Dirichlet process case (approximately, when $\omega=100$). 
On the other hand, the SSE index shows quite different values among the tests: it is well-known that this is a index favoring complex models and  leading to better results when data are over-fitted. 
Therefore, tests with an higher value of $\kappa$ are 
always preferable according to this criterion. 
\begin{table}
\centering
 \caption{Predictive goodness-of-fit indexes for the simulated dataset.}
   \label{tab:indexes} 
\begin{tabular}{l | cc | ccccc }
  \noalign{\smallskip}\footnotesize
Test & $\omega$ & $\kappa$ & SSE & SSAE & WAIC1 & WAIC2 & LPML  \\
   \noalign{\smallskip}\hline\noalign{\smallskip}
A1 &   100 & 0.06  & 6346.59 & 811.16 &-3312.44 &-3312.55 &-3312.55     \\ 
A2 & 4 &  0.09 &  5812.86 & 810.43 & -3312.33 &-3312.42 &-3312.43  \\ 
A3 &2 & 0.1 & 6089.19 & 810.99  &-3312.38	 &-3312.47 &-3312.48 \\ 
A4 & 1.33 & 0.11 &6498.23 & 811.29  &-3312.54 & -3312.62&-3312.63 \\ 
A5 &1.05 & 0.11& \textbf{5725.18}& \textbf{810.39}  & \textbf{-3312.27} &\textbf{-3312.36} &\textbf{-3312.36}  \\ \hline

B1 &100 & 0.43  & 5184.25 & 809.61  &-3311.95 &-3312 &-3312.01 	 \\ 
B2 &4 &  0.67 & 5125.41 & 809.7  &-3312.19 &-3312.25 &-3312.26	 \\ 
B3 & 2 & 0.81 &4610.39 & 809.42  &-3311.92 & -3311.98 & -3312 	 \\ 
B4 &1.33 & 0.93 & 4246.43 & \textbf{809.07} & \textbf{-3311.75} &\textbf{-3311.83} &\textbf{-3311.84}	 \\ 
B5 &1.05 & 1 & \textbf{4571.09} & 809.08  & -3311.96 &-3312.05 &-3312.06\\ \hline

C1 & 100 & 1.56 &3707.5 & 809.36  &\textbf{ -3311.73} & \textbf{-3311.86} &\textbf{-3311.88} 	\\ 
C2 & 4 &  2.67 &2194.1 & 808.8  &-3312.02 &-3312.23 &-3312.26\\ 
C3 & 2 &  3.64 &1223.86 & 809.28  &-3312.62 &-3312.96 &-3312.99\\ 
C4 &1.33 & 5.29 & 748.85 & 808.7  &-3313.05 & -3313.51 &-3313.54 \\ 
C5 & 1.05 & 8.95&\textbf{685}& \textbf{807.96}  &-3312.9 &-3313.36 &-3313.38  \\ 
\noalign{\smallskip}\hline
\end{tabular}
\end{table}

\bigskip
We fitted our model also to a real dataset, the Hidalgo stamps data of \cite{Wilson83} consisting of $n=485$ measurements of 
stamp thickness in millimeters (here multiplied by $10^3$). 
The stamps have been printed between 1872 and 1874 on different paper types, see data histogram in  
Figure~\ref{fig:stamp}.
This dataset has  been analyzed by different authors in the context of mixture models: see, for instance, 
\cite{IzeSom88}, \cite{McaBle06} and \cite{Nieto13}.

We report posterior inference for the set of hyperparameters which is most in agreement with our prior belief: 
the mean distribution is $P_0(d\mu, d\sigma^2)=\mathcal{N}(d\mu; \bar{y}_n, \sigma^2/\kappa_0)\times inv-gamma(d\sigma^2; a, b)$ as before, 
and $\kappa_0=0.005$, $a=2$ and $b=0.1$. The approximation parameter $\varepsilon$ of the  $\varepsilon$-NB$(\omega,\kappa P_0)$ random measure
is fixed to $10^{-6}$; on the other hand, 
in order to set parameters $\omega$ and $\kappa$, we argue as follows:
$\omega$ ranges in $\{1.05, 5, 10, 1000\}$ and we choose the mass parameter $\kappa$ such that the prior mean of  the number of 
clusters, i.e. $\mathbb{E}(K_n)$, is the desired one. 
As noted in Section~\ref{sec:bessel_def}, a closed form of the prior distribution of $K_n$ is not available, so
we resort to Monte Carlo simulation to estimate it. 
Table~\ref{tab:indexes_stamp} shows the four couples of $(\omega, \kappa)$ yielding $\E(K_n)=7$:
indeed, according to \cite{IshJam02} and \cite{McaBle06} and references therein, 
there are at least 7 different groups (but the true number is unknown), corresponding to the
number of types of paper used. 
For an in-depth discussion about the appropriate number of groups in Hidalgo stamps data, 
we refer the reader to \cite{BasMcl97}.
Table~\ref{tab:indexes_stamp} also reports prior standard deviations of $K_n$: even if the a-priori differences are small,  the posteriors appear to be quite different among the 4 tests.
All the posterior distributions on $K_n$ support the conjecture of at least seven distinct modes in the data; 
in particular, Figure~\ref{fig:stamp} (b) 
displays the posterior distribution of $K_n$ for Test 4. A modest amount of mass is given to less than 7 groups, and the mode is in 11.
Even Test 1, corresponding to the Dirichlet process case, does not give mass to less than 7 groups, where 9 is the mode. 
Density estimates seem pretty good; an example is given in Figure~\ref{fig:stamp} (a), 
with 90$\%$ credibility band for Test 4.
\begin{figure}
\centering 
\subfigure[]%
{\includegraphics[width=0.49\textwidth,height=0.45\textwidth]{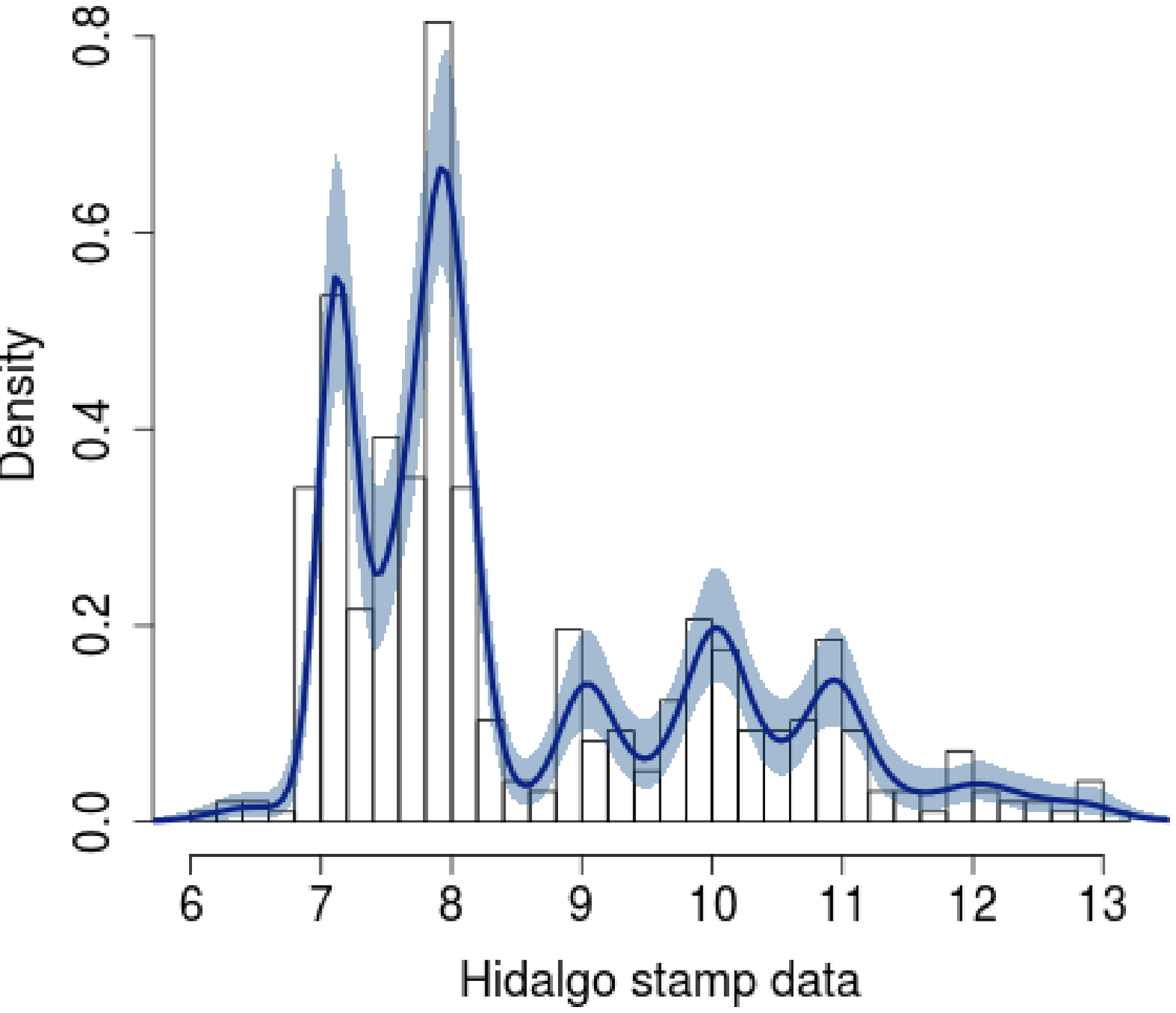}}
\subfigure[]%
{\includegraphics[width=0.49\textwidth,height=0.45\textwidth]{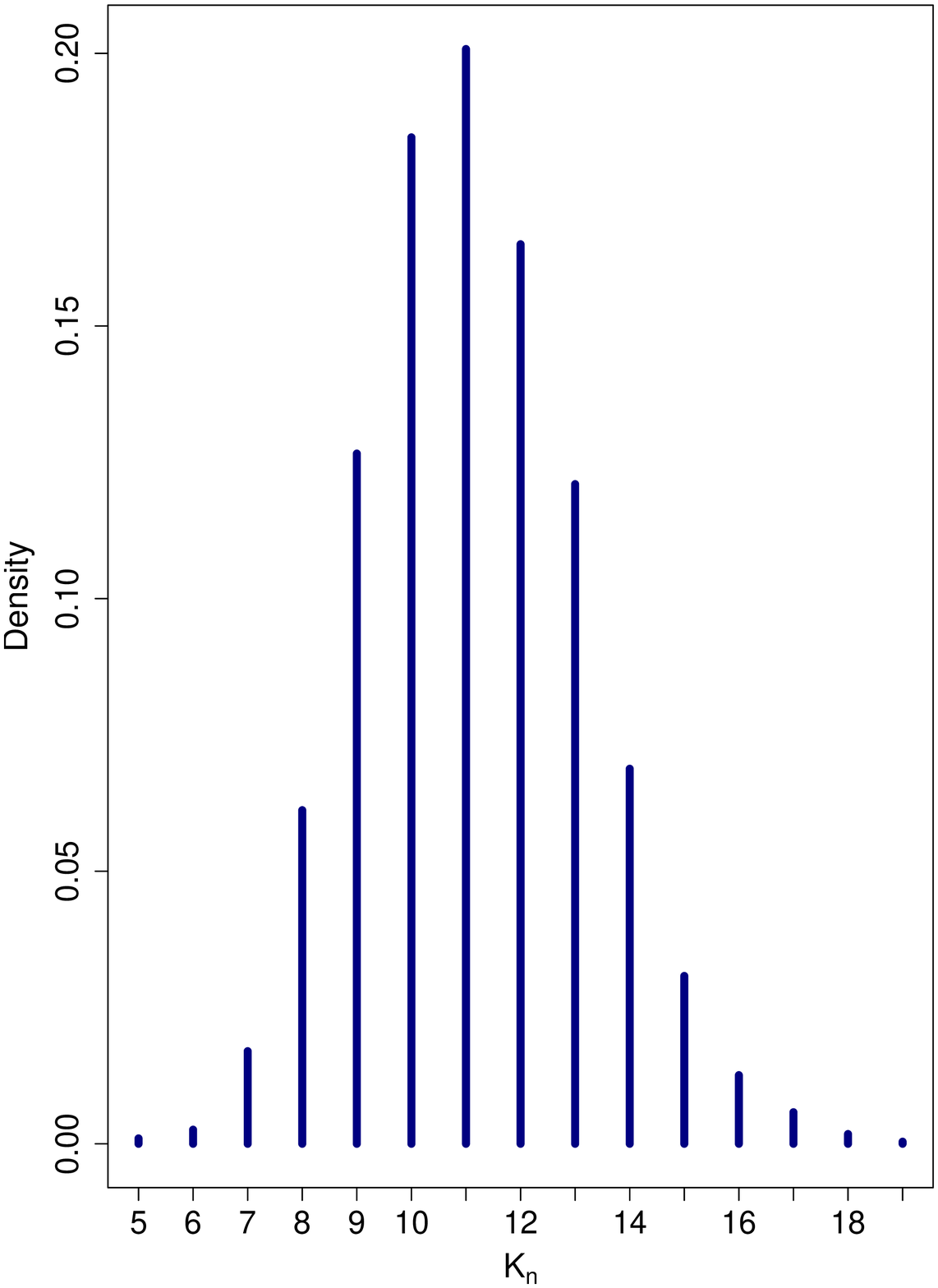}} 
\caption{Posterior inference  for  the Hidalgo stamp data for Test 4: histogram of the data,
density estimate and 90$\%$ pointwise credibility intervals (a); 
posterior distribution of  $K_n$ (b).} 
\label{fig:stamp}
\end{figure}

As in the simulated data example, some predictive goodness-of-fit indexes are reported in Table~\ref{tab:indexes_stamp}: 
the optimal value for each index is indicated in bold.
The SSE is significantly lower when $\omega$ is small, thus suggesting a greater flexibility of the model with small values of $\omega$.
The other indexes assume the optimal value in Test 4 as well, even if those values are similar along the tests.

\begin{table}\centering
 \caption{Predictive goodness-of-fit indexes for the Hidalgo stamps data.}
   \label{tab:indexes_stamp} 
\begin{tabular}{l | cc | cc | ccccc }
  \noalign{\smallskip}\footnotesize
Test & $\omega$ & $\kappa$ & $\mathbb{E}(K_n)$ & $sd(K_n)$ &SSE & SSAE  & WAIC1 & WAIC2 & LPML  \\
   \noalign{\smallskip}\hline\noalign{\smallskip}
1 & 1000 & 0.98  & 7 & 2.04 & 15.17& 384.1 & -713.12 &-713.96 &-714.12     \\ 
2 & 10 &  0.91 &7 & 2.13 &   12.85 & 383.51 & -713.22 &-714.04 &-714.25  \\ 
3 & 5 & 0.92 & 7 & 2.18 &  13.52 & 383.68  &-713.52	 &-714.3 &-714.4 \\ 
4 & 1.05 &  1.02 &7 & 2.32 & \textbf{11.12} & \textbf{383.38}  &\textbf{-712.84} & \textbf{-713.66} &\textbf{-714.05} \\ 
\noalign{\smallskip}\hline
\end{tabular}
\end{table}

\section{Linear dependent NGG mixtures: an application to sports data}
\label{sec:LDNGG}

Let us consider a regression problem, where the response $Y$ is univariate and continuous, for ease of notation. We model the relationship (in distributional terms) between the vector of covariates $\bm
x=(x_1,\dots,x_p)$ and the response $Y$ through a mixture density, where the mixing measure is a collection 
$\left\{ P_{\bm x}, \textbf{x}\in\mathcal{X} \right\}$
of  $\varepsilon$-NormCRMs, being  $\mathcal X$ the space
of all possible covariates. 
We follow the same approach as  in \cite{maceachern1999},  \cite{maceachern2000},
\cite{DeIorio_etal09} for the dependent Dirichlet process.
We define the \emph{dependent $\varepsilon$-NormCRM
process}
$\left\{ P_{\bm x}, {\bm x}\in\mathcal{X} \right\}$, conditionally to $\bm x$, as:
\begin{equation}
  P_{\bm x}\stackrel{d}{=}
  \sum_{j = 0}^{\Neps}P_j\delta_{\bm \gamma_{j}(\bm x) }.
  \label{ed:dep_ngg}
\end{equation}
The  weights $P_j$ are the normalized jumps 
as in \eqref{eq:Pvarepsilon}, while the locations $\gamma_{j}(\bm x)$, $j=1,2,\ldots$, 
are independent stochastic processes with
index set $\calX$ and $P_{0 \bm x}$ marginal distributions. 
Model \eqref{ed:dep_ngg} is such that, marginally, $P_{\bm x}$ follows a $\varepsilon$-NormCRM process, with parameter $(\rho, \kappa P_{0 \bm x})$, where $\rho$ is the intensity of a Poisson process on $\Rea^+$, $\kappa>0$, and $ P_{0 \bm x}$ is a probability on $\Rea$. Observe that, since $\Neps$ and $P_j$ do not depend on $\bm x$, \eqref{ed:dep_ngg} is a generalization of the single weights dependent Dirichlet process \citep[see][for this terminology]{barrientos2012support}. We also assume the functions  ${\bm x}\mapsto\gamma_{j}({\bm x})$ to be continuous.

The dependent $\varepsilon$-NormCRM process in \eqref{ed:dep_ngg}
takes into account the vector of covariates $\bm x$ only through $\gamma_{j}(\bm x)$.  In particular,
when the kernel of the mixture \eqref{eq:modello} belongs to the exponential family,  
for each $j$, $\gamma_j(\bm x)=\gamma(\bm x;\boldsymbol{\tau}_j)$ can be assumed as the link function of a generalized linear model, so that  \eqref{eq:modello} specializes to
\begin{equation}
  \label{eq:hierarc_regression}
\begin{split}
 Y_i| \bm \theta_i,\bm x_i & \ind f(\bm y;\bm \gamma(\bm
x_i,\bm{\theta}_i))\ \ \ i=1,\dots,n\\
  \bm \theta_i|\Peps  &\iid \Peps \ \ \ i=1,\dots,n  \qquad 
\textrm{ where }\Peps\sim \varepsilon-\text{NormCRM}(\rho, \kappa P_0). 
\end{split}
\end{equation} 
This last formulation is convenient because it facilitates parameters interpretation as well as 
numerical posterior computation.

We analyze the Australian Institute of Sport (AIS) data set \citep{CookWeis1994}, which consists of 11 physical measurements on 
202 athletes (100 females and 102 males). Here the response is the lean body mass (lbm), while
three covariates are considered,  the red cell count (rcc),  the height in cm (Ht) and the weight in Kg (Wt).  
The data set is contained in the R package \texttt{DPpackage} \citep{jara2011dppackage}.
The actual model \eqref{eq:hierarc_regression} we consider here is when $f(\cdot;\mu,\eta^2)$ is the Gaussian distribution with $\mu$ mean and $\eta^2$ variance; 
moreover, $\mu=\gamma({\bm x}, \bm\theta)= {\bm x}^t \bm\theta$, and the mixing measure $\Peps$ is the $\varepsilon$-NGG$(\kappa, \sigma,P_0)$, as introduced in \cite{peps}. We have considered two cases, when mixing the variance $\eta^2$ with respect to the NGG process
or when the variance   $\eta^2$  is given a parametric density; in both cases, by linearity of the mean ${\bm x}^t \bm\theta$, the model (here called linear dependent NGG mixture) can be interpreted as a NGG process mixture model, and inference can be achieved via an algorithm similar to that in Section~\ref{sec:epsHNRMImixtures}.
We  set $\varepsilon=10^{-6}$, $\sigma\in\{ 0.001,0.125,0.25\}$, and
$\kappa$ such that $\E(K_n)\simeq5$ or 10. When the variance $\eta^2$ is included in the location points of the $\varepsilon$-NGG process, then $P_0$ is $\calN_4({\bm b}_0,\Sigma_0)\times$inv-gamma$(\nu_0/2, \nu_0\eta_0^2/2)$;
on the other hand, when  $\eta^2$  is given a parametric density, then $\eta^2\sim$inv-gamma$(\nu_0/2, \nu_0\eta_0^2/2)$. 
We fixed hyperparameters 
in agreement with the 
least squares estimate: ${\bm b}_0=(-50,5,0,0)$, $\Sigma_0=diag(100,10,10,10)$, 
$\nu_0=4$, $\eta_0^2=1$.  For all the experiments, we computed the
posterior of the number of groups, the predictive densities at different
values of the covariate vectors  and the cluster
estimate via posterior maximization of Binder's loss function \citep[see][]{Lau_Green07}. 
Moreover, we compared the
different prior settings computing predictive goodness-of-fit tools, specifically log pseudo-marginal likelihood (LPML) and the 
sum of squared errors (SSE), as introduced in Section~\ref{sec:Bessel_application}.
The minimum value of SSE, among our experiments, was achieved when $\eta^2$ is included in the location of the $\varepsilon$-NGG process,  $\sigma=0.001$ and 
$\kappa=0.8$ so that $\E(K_n)\simeq5$. On the other hand, the  optimal LPML
was achieved when $\sigma=0.125$, $\kappa=0.4$, and  $\E(K_n)\simeq5$.
\begin{figure}[h!]
\centering 
\subfigure[]%
{\includegraphics[width=0.49\textwidth,height=0.5\textwidth]{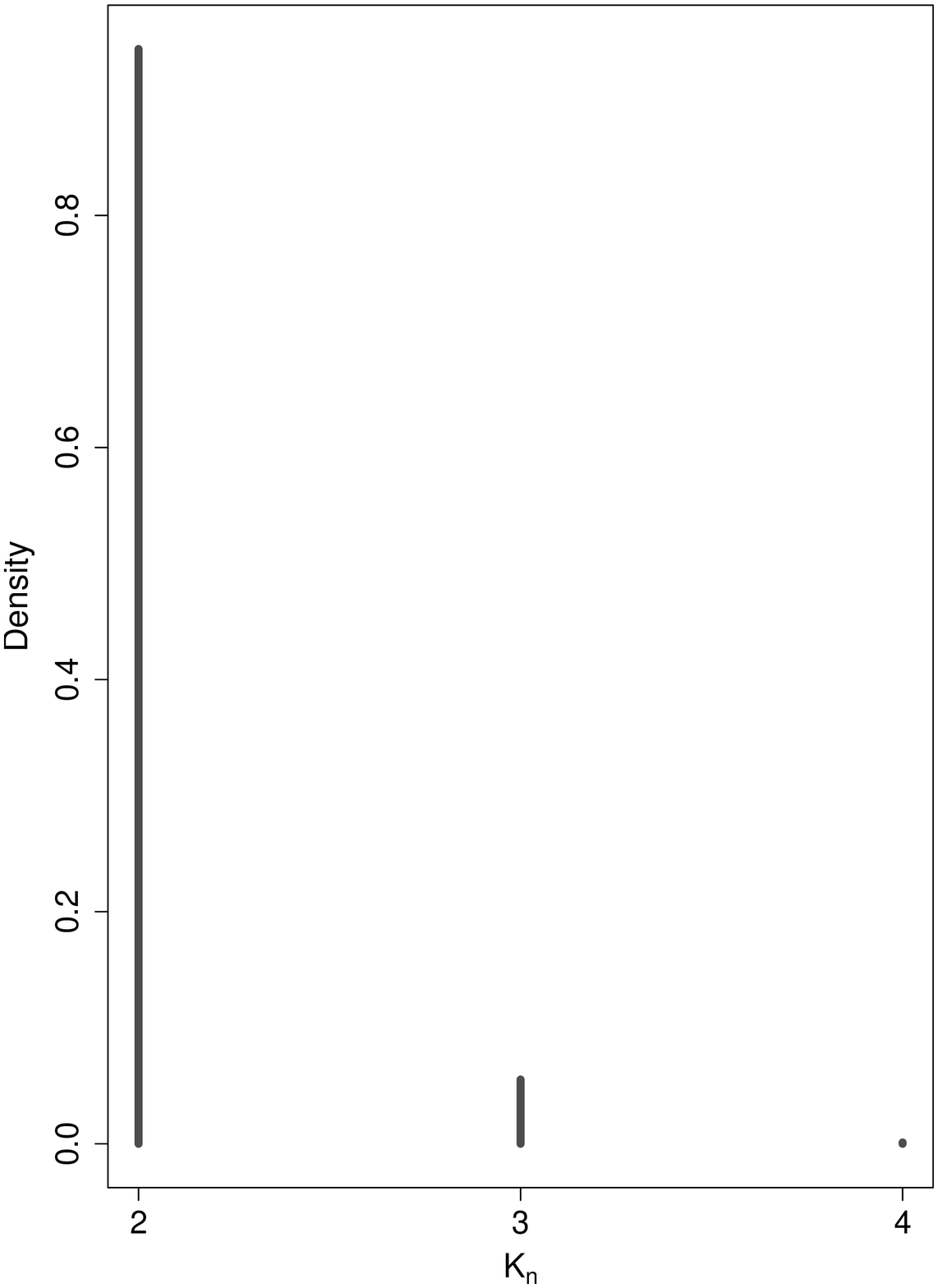}}
\subfigure[]%
{\includegraphics[width=0.49\textwidth,height=0.5\textwidth]{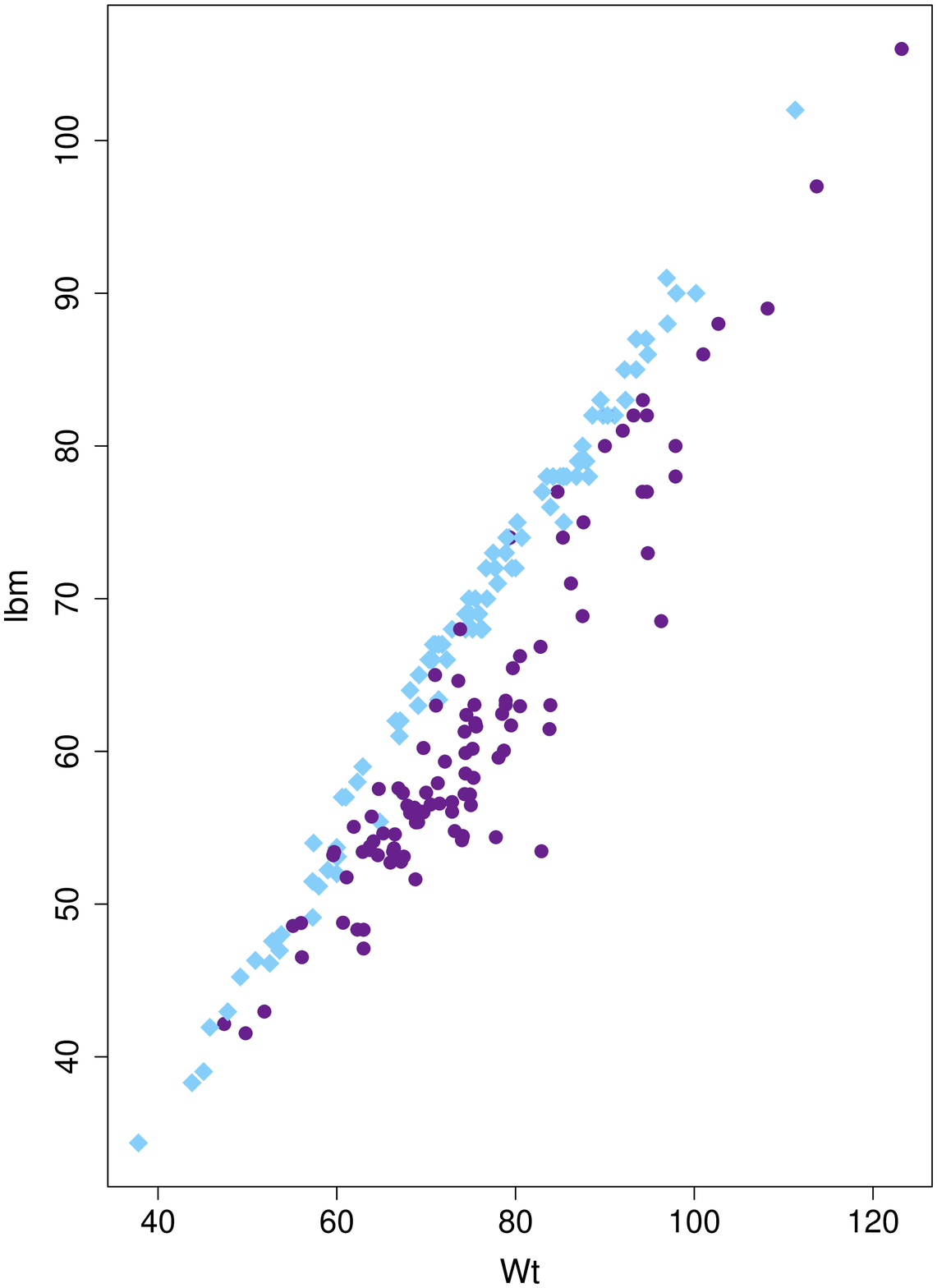}} 
\caption{Posterior distributions of the number $K_n$ of groups (a) and cluster estimate  (b) under the linear dependent $\varepsilon-$NGG mixture.} 
\label{fig:postKn_clusters}
\end{figure}
Posterior of $K_n$ and cluster estimate under this last hyperparameter setting
are in Figure~\ref{fig:postKn_clusters} ($(a)$ and ($b$), respectively);
in particular the cluster estimate is displayed in the scatterplot of the
Wt vs lbm. In spite of the vague prior, the posterior of $K_n$ is  almost
degenerate on $2$, giving evidence to the existence of two linear relationships between lbm and Wt. 

Finally, Figure~\ref{fig:predictives} displays  predictive densities and 95\% credibility bands for 3 athletes, a female 
(Wt=60, rcc=3.9, Ht=176 and lbm=53.71), 
and two males (Wt=67.1,113.7, rcc=5.34,5.17, Ht=178.6, 209.4 and lbm=62,97, respectively); 
the dashed lines are observed values of the response. 
Depending on the value of the covariate, the distribution shows one or two peaks: 
this reflects the dependence of the grouping of the data on the value of $\mathbf{x}$.	
\begin{figure}[h!]
\centering 
\subfigure[]%
{\includegraphics[width=0.32\textwidth,height=0.43\textwidth]{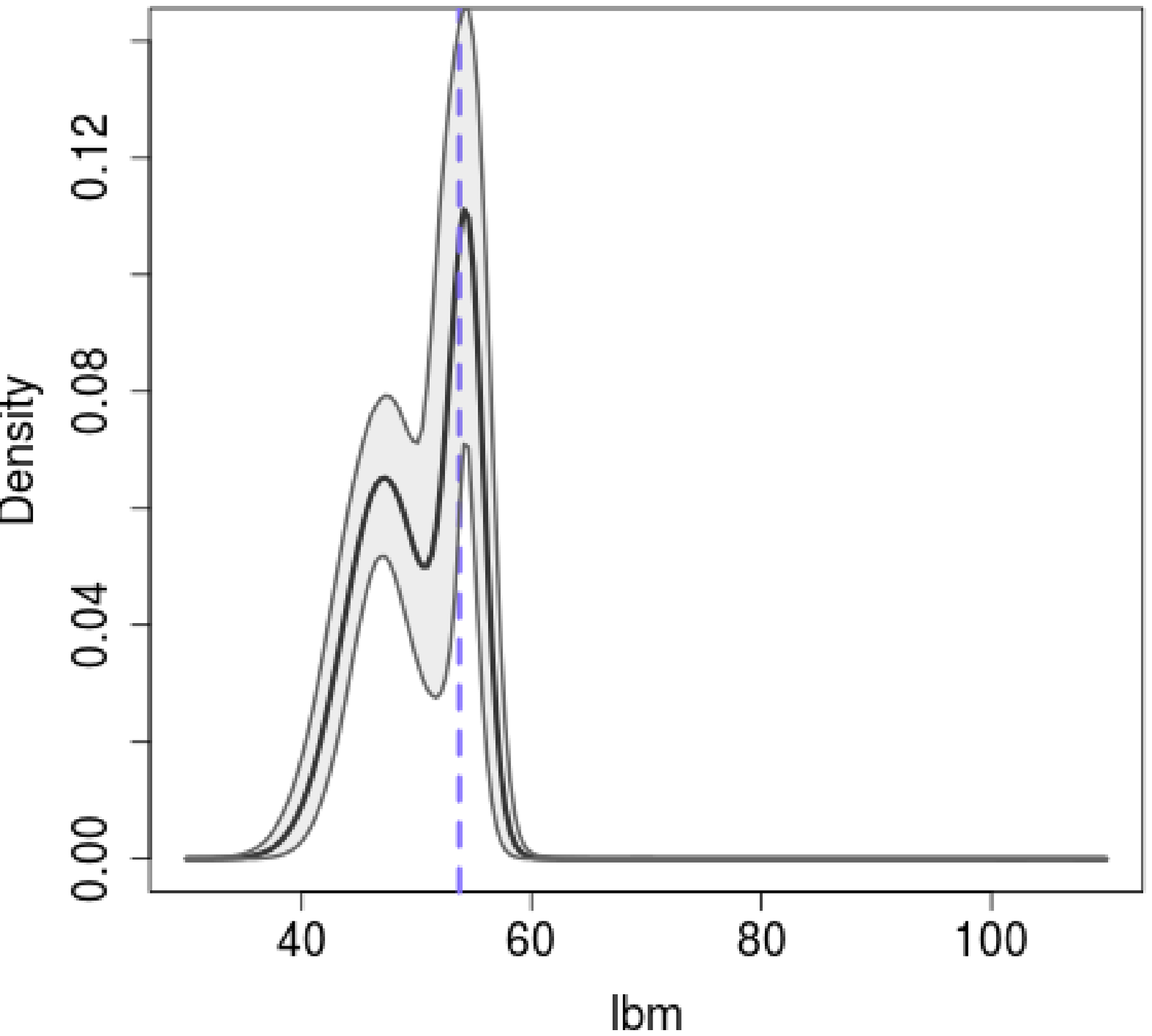}}
\subfigure[]%
{\includegraphics[width=0.32\textwidth,height=0.43\textwidth]{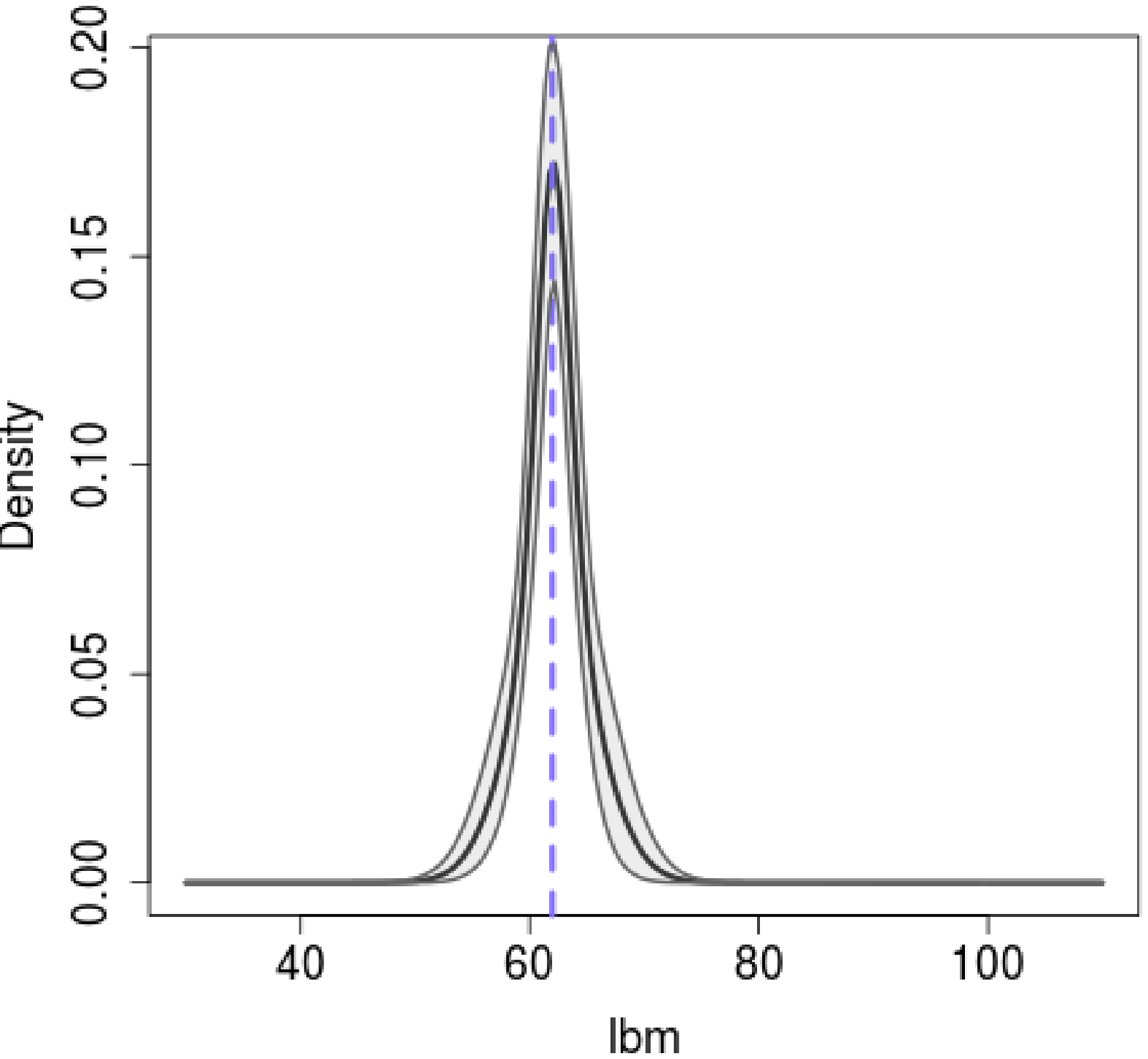}} 
\subfigure[]%
{\includegraphics[width=0.32\textwidth,height=0.43\textwidth]{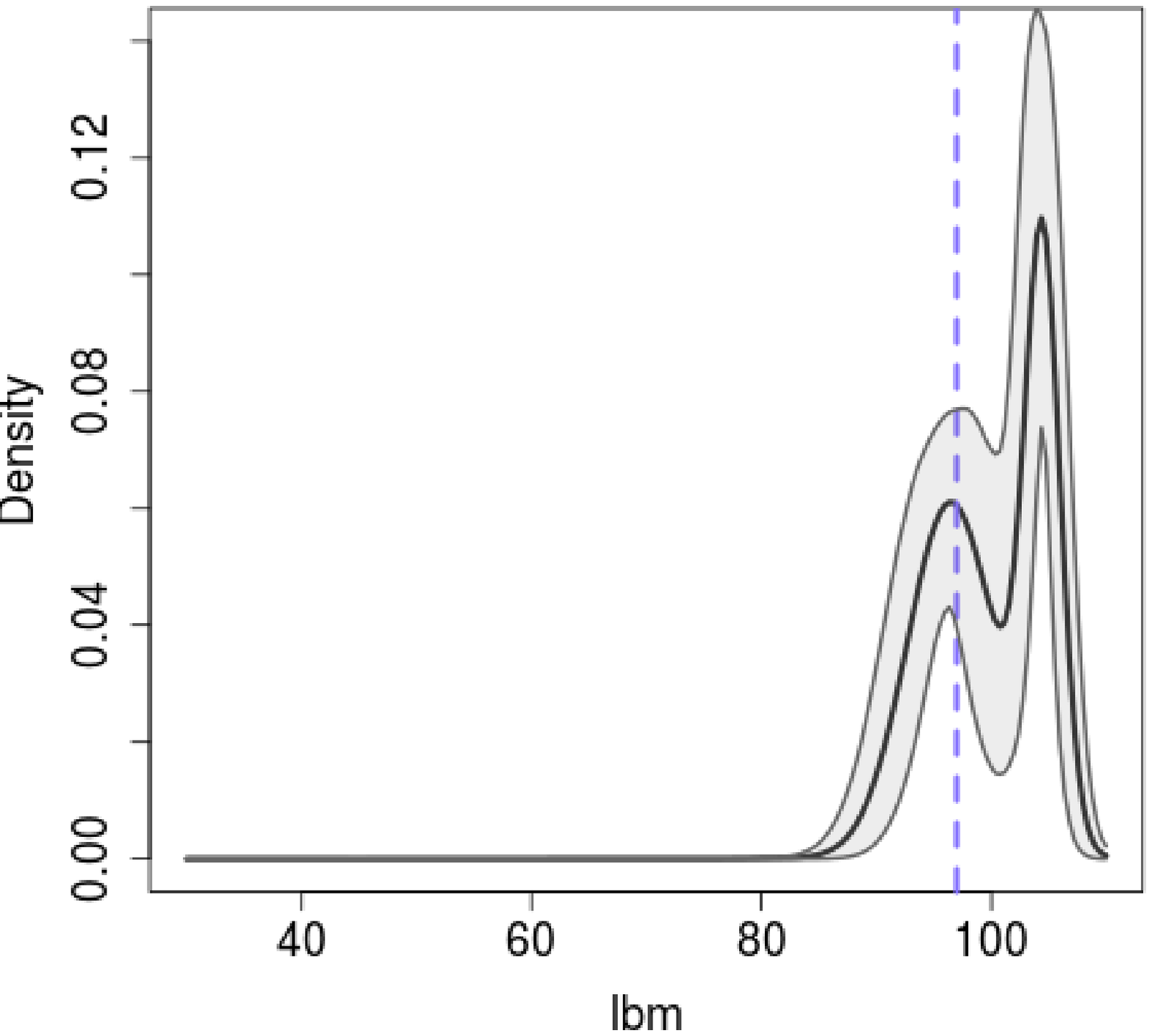}} 
\caption{Predictive distributions of lbm for three different athletes: Wt=60, rcc=3.9, Ht=176 (a),
Wt=67.1, rcc=5.34, Ht=178.6  (b),  Wt=113.7, rcc=5.17, Ht=209.4  (c). The shaded area is the predictive 95\% pointwise credible interval, while the dashed vertical line denotes the observed value of the response.} 
\label{fig:predictives}
\end{figure}
This figure highlights the versatility of nonparametric priors in a linear regression setting with respect to 
the customary parametric priors: indeed, the model is able to capture in detail the behavior of the data, even when several clusters are present.

\section{Discussion}
\label{sec:last}
We have proposed a new model for density and cluster estimation in the Bayesian nonparametric framework.
In particular,  a finite dimensional process, the  $\varepsilon$-NormCRM, has been defined, which  
converges in distribution to the corresponding normalized completely random measure, when $\varepsilon$ tends to 0. Here, 
the $\varepsilon$-NormCRM is the mixing measure in a mixture model. In this paper we have fixed $\varepsilon$ 
very small, but we could choose a prior for $\varepsilon$ and include this parameter into the Gibbs sampler scheme.
Among the achievements of the work, we have generalized all the theoretical results obtained in the special case 
of NGG in \cite{peps}, including the expression of the eppf for an $\varepsilon$-NormCRM process, 
its convergence to the corresponding eppf of the nonparametric underlying process and the posterior characterization of $P_\varepsilon$. 
Moreover, we have provided a general Gibbs Sampler scheme to sample from the posterior of the mixture model.
To show the performance of our algorithm and the flexibility of the model, we
have illustrated two examples via normalized completely random measure mixtures:
in the first application, we have introduced a new normalized completely random measure, named normalized Bessel random measure; we have studied its theoretical properties 
and used it as the mixing measure in a model to fit simulated and real datasets.
The second example we have dealt with is a linear dependent $\varepsilon$-NGG mixture, where the dependence lies on the support 
points of the mixing random probability, to fit a well known dataset.
Current and future research is devoted on the use of our approximation on more complex dependence structures. 


\newpage
\begin{center}\textbf{
  APPENDIX: DETAILS ON FULL-CONDITIONALS FOR THE GIBBS SAMPLER }
\end{center}

Here, we provide some details about Step 3 of the Gibbs Sampler in Section~\ref{sec:epsHNRMImixtures}.
As far as Step 3a is concerned, the full-conditional $\mathcal{L}(\varepsilon | u, \boldsymbol\theta)$ is obtained 
integrating out $N_\varepsilon$ (or equivalently $N_{na}$) from the law $ \mathcal{L}(N_\varepsilon, u,\bm\theta)$, as follows:
\begin{equation*}
 \begin{split}
\mathcal{L}(\varepsilon| u, \boldsymbol{\theta}, \boldsymbol{Y}) &\propto \sum_{N_{na}=0}^{+\infty}
  \mathcal{L}(N_{na},\varepsilon, u, \boldsymbol{\theta}, \boldsymbol{Y})  \\
  & = \sum_{N_{na}=0}^{+\infty} \pi(\varepsilon)e^{-\Lambda_\varepsilon}\dfrac{\Lambda_{\varepsilon,u}^{N_{na}}}{\Lambda_\varepsilon}
  \dfrac{(N_{na}+k)}{N_{na}!}\prod_{i=1}^k 
  \int_\varepsilon^{+\infty}\kappa s^{n_i}\e^{-u s}\rho(s)ds\\
&= \left( \prod_{i=1}^k\int_\varepsilon^{+\infty}\kappa s^{n_i} \e^{-u s}\rho(s)ds\right) e^{\Lambda_{\varepsilon,u}-\Lambda_\varepsilon}
\dfrac{\Lambda_{\varepsilon,u}+k}{\Lambda_{\varepsilon}}\pi(\varepsilon) = 
f_\varepsilon(u; n_1, \dots, n_k)\pi(\varepsilon),
 \end{split}
\end{equation*}
where we used the identity $\sum_{N_{na}=0}^{+\infty}\Lambda_{\varepsilon,u}^{N_{na}}
  (N_{na}+k)/(N_{na}!)$ $=e^{\Lambda_{\varepsilon,u}}(\Lambda_{\varepsilon,u}+k)$, as previously noted.
 Moreover, $f_\varepsilon(u; n_1, \dots, n_k)$ is defined in \eqref{eq:integr_eppf}. This step depends explicitly on the expression of $\rho(s)$.
  
Step 3.b consists in sampling from $\mathcal{L}(P_\varepsilon|\varepsilon, u, \boldsymbol\theta)$ and has already been described 
in the proof of Propositùion~\ref{prop:Posterior}. However, for a complete outline of the algorithm, we list the full-conditionals resulting into Step 3b:
 \begin{enumerate}
  \item [(i).] $\mathcal{L}(N_{na}|\varepsilon, \bm Y, u, \bm \theta) = \dfrac{\Lambda_{\varepsilon u}}{\Lambda_{\varepsilon u}+k}
  \mathcal{P}_1(\Lambda_{\varepsilon u})+\dfrac{k}{\Lambda_{\varepsilon u}+k}\mathcal{P}_0(\Lambda_{\varepsilon u})$; this is formula \eqref{eq:fullcond_Nna}.
  \item [(ii).] Non-allocated jumps: iid from 
  $\mathcal{L}(J_j) \propto e^{-uJ_j}\rho(J_j)\mathbbm{1}_{(\varepsilon, \infty)}(J_j),$ $j =1,\dots,N_{na}$; 
  see the second factor of the last expression in \eqref{eq:postAugm}.
  \item [(iii).] Allocated jumps: iid from $\mathcal{L}(J_{l^*_i}) \propto J_{l^*_i}^{n_i}e^{-uJ_{l^*_i}}\rho(J_{l^*_i})
  \mathbbm{1}_{(\varepsilon, \infty)}(J_{l^*_i})$, $i=1, \dots, k$;  
  see the first factor of the last expression in \eqref{eq:postAugm}.
  \item [(iv).] Non-allocated points of support: iid from $P_0$; see \eqref{eq:leggecongiuntamodelloAPP}.
  \item [(v).] Allocated points of support: iid from  $\mathcal{L}(\tau^*_i) \propto \{ \prod_{j \in C_i} k(X_j;\tau_i) \} P_0(\tau_i)$, 
  $i = 1, \dots, k$; see \eqref{eq:leggecongiuntamodelloAPP}.
 \end{enumerate}

 

\end{document}